\documentclass{amsart}
\input{xy}
\xyoption{poly}
\xyoption{2cell}
\xyoption{all}
\usepackage{bm}
\usepackage{amssymb,latexsym,xcolor}
\usepackage{amsmath,amsthm}
\usepackage[utf8]{inputenc}
\usepackage[english]{babel}
\usepackage{amsrefs}
\usepackage{tkz-euclide}
\usepackage[mathscr]{eucal}
\usepackage[all]{xy}
\usepackage{tikz-cd}
\usetkzobj{all} 
\usepackage{tikz,pgf}
\usetikzlibrary{matrix}
\usetikzlibrary{shapes.geometric,calc}
\usepackage{adjustbox}
\usepackage{stackrel}
\usepackage{hyperref}
\usepackage{extarrows}
\newtheorem{theorem}{Theorem}[section]
\newtheorem{lemma}[theorem]{Lemma}
\newtheorem{proposition}[theorem]{Proposition}

\newtheorem{corollary}[theorem]{Corollary}

\theoremstyle{definition}
\newtheorem{definition}[theorem]{Definition}
\newtheorem{example}[theorem]{Example}

\theoremstyle{remark}

\numberwithin{equation}{section}

\DeclareMathOperator{\prin}{prin}

\DeclareMathOperator{\repr}{\mathscr{P}\mbox{-spr}}
\DeclareMathOperator{\rep}{rep}

\DeclareMathOperator{\dm}{\underline{cdim}}
\DeclareMathOperator{\supp}{supp}
\DeclareMathOperator{\md}{mod}

\DeclareMathOperator{\Hom}{Hom}
\DeclareMathOperator{\Supp}{supp}
\DeclareMathOperator{\csupp}{csupp}
\def\P{\mathscr{P}}

\newcommand\colorb[1]{\color{red}{\bm #1}}

\newcommand\smrud[1]{%
  \begingroup\setlength\arraycolsep{0pt}\renewcommand{\arraystretch}{0.5}\Huge\ensuremath{\begin{matrix}  #1\end{matrix}\arrow[rd]\arrow[ru]\endgroup}}
\newcommand\sm[1]{%
  \begingroup\setlength\arraycolsep{0pt}\renewcommand{\arraystretch}{0.5}\Huge\ensuremath{\begin{matrix}#1\end{matrix}\endgroup}}
\newcommand\smru[1]{%
  \begingroup\setlength\arraycolsep{0pt}\renewcommand{\arraystretch}{0.5}\Huge\ensuremath{\begin{matrix}#1\end{matrix}\arrow[ru]\endgroup}}
\newcommand\smrd[1]{%
  \begingroup\setlength\arraycolsep{0pt}\renewcommand{\arraystretch}{0.5}\Huge\ensuremath{\begin{matrix}#1\end{matrix}\arrow[rd]\endgroup}}
\newcommand\smr[1]{%
  \begingroup\setlength\arraycolsep{0pt}\renewcommand{\arraystretch}{0.4}\Huge\ensuremath{\begin{matrix}#1\end{matrix}\arrow[r]\endgroup}}
\newcommand{\gon}[4]{\begin{tikzpicture}
\tkzDefPoint(0,0){x}\tkzDefPoint(#1,0){A}
\tkzDefPointsBy[rotation=center x angle 360/10](A,B,C,D,E,F,G,H,I){B,C,D,E,F,G,H,I,J}
\tkzDrawPoints[fill=blue,size=3,color=black](A,B,C,D,E,F,G,H,I,J)
\tkzDrawPolygon[ultra thin](A,B,C,D,E,F,G,H,I,J)
\tkzDrawPolygon[thin,color=blue](#2,#3)
\tkzLabelSegment[above](#2,#3){#4}
\end{tikzpicture}
}
\newcommand{\gonred}[4]{\begin{tikzpicture}
\tkzDefPoint(0,0){x}\tkzDefPoint(#1,0){A}
\tkzDefPointsBy[rotation=center x angle 360/10](A,B,C,D,E,F,G,H,I){B,C,D,E,F,G,H,I,J}
\tkzDrawPoints[fill=blue,size=3,color=black](A,B,C,D,E,F,G,H,I,J)
\tkzDrawPolygon[red](A,B,C,D,E,F,G,H,I,J)
\tkzDrawPolygon[thin,color=blue](#2,#3)
\tkzLabelSegment[above](#2,#3){#4}
\end{tikzpicture}
}
\newcommand{\gone}[4]{\begin{tikzpicture}
\tkzDefPoint(0,0){x}\tkzDefPoint(#1,0){A}
\tkzDefPointsBy[rotation=center x angle 360/10](A,B,C,D,E,F,G,H,I){B,C,D,E,F,G,H,I,J}
\tkzDrawPolygon[thin,color=blue](#2,#3)
\tkzLabelSegment[left](#2,#3){#4}
\end{tikzpicture}
}

\newcommand{\gonu}[4]{\begin{tikzpicture}
\tkzDefPoint(0,0){x}\tkzDefPoint(#1,0){A}
\tkzDefPointsBy[rotation=center x angle 360/9](A,B,C,D,E,F,G,H){B,C,D,E,F,G,H,I}
\tkzDrawPoints[fill=blue,size=3,color=black](A,B,C,D,E,F,G,H,I)
\tkzDrawPolygon(A,B,C,D,E,F,G,H,I)
\tkzDrawPolygon[thin,color=blue](#2,#3)
\tkzLabelSegment[above](#2,#3){#4}
\end{tikzpicture}
}

\newcommand{\gonured}[4]{\begin{tikzpicture}
\tkzDefPoint(0,0){x}\tkzDefPoint(#1,0){A}
\tkzDefPointsBy[rotation=center x angle 360/9](A,B,C,D,E,F,G,H){B,C,D,E,F,G,H,I}
\tkzDrawPoints[fill=blue,size=3,color=black](A,B,C,D,E,F,G,H,I)
\tkzDrawPolygon[red](A,B,C,D,E,F,G,H,I)
\tkzDrawPolygon[thin,color=blue](#2,#3)
\tkzLabelSegment[above](#2,#3){#4}
\end{tikzpicture}
}
\newcommand{\gonuo}[4]{\begin{tikzpicture}
\tkzDefPoint(0,0){x}\tkzDefPoint(#1,0){A}
\tkzDefPointsBy[rotation=center x angle 360/9](A,B,C,D,E,F,G,H){B,C,D,E,F,G,H,I}
\tkzDrawPolygon[thin,color=blue](#2,#3)
\tkzLabelSegment[above](#2,#3){#4}
\end{tikzpicture}
}

\begin{document}

\title[A geometric realization of socle-projective categories...]{A geometric realization of socle-projective categories for  posets of type $\mathbb{A}$}

\author{Ralf Schiffler}
\address{Department of Mathematics, University of Connecticut, Storrs, CT 06269-3009, USA}
\email{schiffler@math.uconn.edu}
\thanks{The first author was supported  by NSF CAREER Grant DMS-1254567, NSF Grant DMS-1800860 and by the
University of Connecticut.}



\author{Robinson-Julian Serna}
\address{School of Mathematics and Statistics, Pedagogical and Technological University of Colombia, Tunja, Boyacá 150001}
\email{robinson.serna@uptc.edu.co}
\thanks{The second author was supported in part by Colciencias Conv. 727.  He would  like to thank Department of Mathematics at University of Connecticut for hospitality and support  during his visit in Fall 2018.}



\dedicatory{To the memory of A.G. Zavadskij  (1946–2012)}

\keywords{category of diagonals, cluster category, cluster algebra, poset of type $\mathbb{A}$, socle-projective representation, Auslander Reiten quiver.}

\begin{abstract}
This paper establishes a link between the theory of cluster  algebras and the theory of representations of partially ordered sets.  We introduce  a class of posets   by requiring avoidance of  certain types of peak-subposets and show that these posets can be realized as the posets of quivers of type $\mathbb{A}$ with certain additional arrows.  This class of posets is therefore called \emph{posets of type $\mathbb{A}$}.  We then give a geometric realization of the category of finitely generated socle-projective modules over the incidence algebra  of a poset of type $\mathbb{A}$  as a combinatorial category of certain diagonals of a regular polygon. This construction is inspired by the realization of the cluster category of type $\mathbb{A}$ as the category of all diagonals by Caldero, Chapoton and the first author  \cite{schiffler1}.\\ 

We also study the subalgebra of the cluster algebra generated by those cluster variables that correspond to the socle-projectives under the above construction. We give a sufficient condition for when this subalgebra is equal to the whole cluster algebra.
\end{abstract}
\maketitle
\tableofcontents
\section{Introduction}
Geometric realizations  of  algebraic structures using the combinatorial geometry of  surfaces have been developed by different authors in  recent years (for instance  see \cites{karin1,karin2,karin3,karin4, schiffler1,brustle,David-Roesler,Demonet,Demonet1,opper,rschiffler,BGMS}). 
This approach provides geometric and  combinatorial tools for the study of the objects and morphisms in the category. It plays  an important role in cluster-tilting theory and in representation theory in general.  For example, the category $\mathcal{C}$ of diagonals (not including boundary edges) in a regular polygon  $\Pi_{n+3}$ with $n+3$ vertices  introduced by  Caldero, Chapoton and the first author \cite{schiffler1} is a geometric realization of the  cluster category  of type  $\mathbb{A}_n$; which, in greater generality, was defined simultaneously by Buan-Marsh-Reiten-Reineke-Todorov \cite{BMRRT}. They defined cluster categories as  orbit categories of the bounded derived category of hereditary algebras. As an application in \cite{schiffler1}, the module category  of a cluster-tilted algebra of type $\mathbb{A}_n$  is  
described by a category of diagonals $\mathcal{C}_T$ in  $\Pi_{n+3}$, where $T$ is a triangulation of  $\Pi_{n+3}$. \\

The present work  links  the theory of cluster algebras \cite{fomin} and cluster categories with the theory of  representations  of  partially ordered sets (in short, posets)  through a  geometric realization inspired by the one in \cite{schiffler1}. \\

The representation theory of posets  was established parallel to the development of the representation theory of Artin algebras; the notion of a matrix representation of a poset $\P$ over an  algebraically closed field $k$ was introduced in the 1970s by Nazarova and  Roiter \cite{nazarova}. Aside from matrix representations of a poset $\P$, the concept of $\P$-space (or representation of $\P$) over  a field  $k$ was introduced by Gabriel \cite{gabriel1}  in connection with the investigation  of representations  of quivers.  If $(\P,\preceq)$ is a finite poset, the category of $\P$-spaces of the poset $\P$ is nothing else than the category of socle-projective modules of the incidence algebra $k\P^{\star}$ of the enlarged poset $\P^{\star}=\P\cup\{\star\}$ such that $x\prec \star$ for each $x\in\P$;  however, there are  genuine methods in representation theory of posets such as the differentiation algorithms \cite{Zavadskijp,canadas1,canadas2,simson}. In a more general situation, Simson studied the category of peak $\P$-spaces which is identified with the category of socle-projective $k\P$-modules, where $k\P$ is the incidence algebra of a poset $\P$  \cite{simson3,simson4}.  He gave the finiteness criterion for those categories while his student J. Kosakowska classified the sincere posets of finite representation type \cite{justina1,justina2,justina}. Moreover, the tameness criterion was given by Kasjan and Simson in \cite{kasjan}. In general, the theory of representations of posets plays an important role in the study of lattices over orders, in the classification of indecomposable lattices over  some simple curve  singularities and in the classification of abelian groups of finite rank (see \cites{Arnold,simson}).\\

In this paper, we introduce  a  class of  posets  which we call  posets of type $\mathbb{A}$. Roughly speaking, they are posets with $n\geq 1$ elements  whose category of socle-projective representations is embedded in the category of representations of a Dynkin quiver of type $\mathbb{A}_n$. We characterize these posets as those not allowing a peak-subposet of one of four types, see Definition~\ref{defposetypeA}. Then, we define a subcategory $\mathcal{C}_{(T,F)}$ of the category  $\mathcal{C}_T$ of diagonals  of a triangulated polygon $\Pi_{n+3}$ with $n+3$ vertices to give a geometric realization of  the category of socle-projective representations 
 $\md_{sp}(k\P)$ of posets $\P$ of type $\mathbb{A}$, where $T$ is a triangulation of $\Pi_{n+3}$ associated to a Dynkin quiver $Q$ of type $\mathbb{A}_n$ and $F$ is a set of additional arrows for $Q$.  We show that there is an equivalence of categories $\mathcal{C}_{(T,F)}\to\md_{sp}(k\P)$ in Theorem~\ref{omega}.  Moreover, we define a subalgebra $\mathcal{A}(\P)$  of the cluster algebra $\mathcal{A}=\mathcal{A}(\bm x, Q)$ generated by the cluster variables associated to diagonals in $\mathcal{C}_{(T,F)}$ and diagonals in $T$; then,  we establish that if $\P$ is the poset whose Hasse quiver is a Dynkin quiver $Q$ of type $\mathbb{A}_n$ then  $\mathcal{A}=\mathcal{A}(\P)$  in Theorem~\ref{subalgebra} .\\

The paper is organized as follows:  In section \ref{preliminaries}, we recall some notation and results about  categories of diagonals in regular polygons  and  categories of socle-projective representations of posets. In section 3, we define and study  posets of type $\mathbb{A}$. Section \ref{Catdiagonals} is devoted to proving our main result,  Theorem \ref{omega}. Finally, the last section deals with the subalgebras $\mathcal{A}(\P)$ of  the cluster algebra $\mathcal{A}$.

\section{Preliminaries} \label{preliminaries}
\subsection{Category of diagonals $\mathcal{C}_T$ }\label{Catdiagonals0}

We recall some results and notation of \cite{schiffler1} (see also Chapter $3$ in \cite{schiffler}) which are used in this work. A \textit{diagonal} in a  regular  polygon  is a straight line segment that joins two of the vertices and goes through the interior of the polygon.  A \textit{triangulation} of the polygon is a maximal set of non-crossing
diagonals. Such a triangulation cuts the polygon into triangles.\\

Let  $T=\lbrace \tau_1,\dots,\tau_n\rbrace$  be a triangulation of a regular polygon $\Pi_{n+3}$ (or $(n+3)$-gon) with $n+3$ vertices  and let $\gamma$ and $\gamma'$ be diagonals that are not in $T$.  The diagonal $\gamma$ is related to the diagonal $\gamma'$ by a \textit{pivoting elementary move} if they share a vertex on the boundary (this vertex is called  \textit{pivot}), the other vertices of $\gamma$ and $\gamma'$ are the vertices of a boundary edge of the polygon and the rotation around the pivot is positive (for the trigonometric direction) from $\gamma$  to $\gamma'$. Let $P_{v}:\gamma\to\gamma'$ denote the pivoting elementary move from $\gamma$ to $\gamma'$ with pivot $v$. Compositions  of pivoting elementary moves are called \textit{pivoting paths}.   \\

The combinatorial $k$-linear additive category $\mathcal{C}_T$ of diagonals is defined as follows: The objects are positive integral linear combinations   of diagonals that are not in $T$. By additivity, it is enough define morphisms between diagonals. To do that, we recall that  the \textit{mesh relations} are the equivalence relation between pivoting paths  induced by  identifying  every couple  of pivoting paths of the form  $$\gamma\xrightarrow{P_{v_1}}\beta\xrightarrow{P_{v'_2}}\gamma'=\gamma\xrightarrow{P_{v_2}}\beta'\xrightarrow{P_{v'_1}}\gamma'$$ 
where $v_1\neq v'_2$ and $v_2\neq v'_1$ (see Figure \ref{mesh0}). In these relations, diagonals in $T$ or boundary edges are allowed with the  following convention: If one of the intermediate edges ($\beta$ or $\beta'$) is either boundary edge or diagonal in $T$, the corresponding term in the mesh relation is replaced by zero. Thus, the space of morphisms from a diagonal $\gamma\notin T$ to a diagonal $\gamma'\notin T$ is the quotient of the vector space over $k$ spanned by pivoting paths from $\gamma$ to $\gamma'$ modulo the  mesh relations.\\

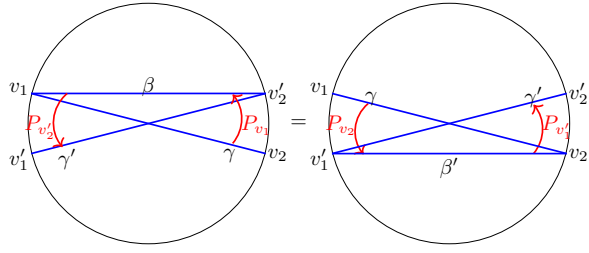
\begin{figure}[h]
\begin{adjustbox}{scale=0.8,center}
\begin{tikzpicture}
    
    \draw (0,0) circle (2cm) ;
\draw (5,0) circle (2cm);  
\node (0) at (2.5,0) {$=$};
  
   \draw [line width = 0.8pt, blue] (-1.95,.5) -- (1.95,-.5);
  \draw [line width = 0.8pt, blue] (-1.95,-0.5) -- (1.95,0.5);

 \draw [line width = 0.8pt, blue] (-1.95,.5) -- (1.95,0.5);
 

\draw (-2.15,.8) node[below] {$v_1$};
\draw (2.15,-.3) node[below] {$v_2$};
\draw (1.35,-.3) node[below] {$\gamma$};
\draw (-1.35,-.3) node[below] {$\gamma'$};

\draw (-2.15,-.3) node[below] {$v'_1$};
\draw (2.15,0.8) node[below]{$v'_2$};
\draw (0,.9) node[below] {$\beta$};

\draw [->, red, thick] (1.4,-0.33) arc (-40:40:18pt);
\draw [->, red, thick] (-1.35,0.5) arc (130:220:18pt);
\draw (1.8,0.25) node[below, red] {$P_{v_1}$};
\draw (-1.8,0.25) node[below, red] {$P_{v'_2}$};

\draw [->, red, thick] (6.4,-0.5) arc (-40:40:18pt);
\draw [->, red, thick] (3.65,0.33) arc (130:220:17pt);
\draw (6.8,0.25) node[below, red] {$P_{v'_1}$};
\draw (3.2,0.25) node[below, red] {$P_{v_2}$};


   \draw [line width = 0.8pt, blue] (3.05,.5) -- (6.95,-.5);
  \draw [line width = 0.8pt, blue] (3.05,-0.5) -- (6.95,0.5);

 \draw [line width = 0.8pt, blue] (3.05,-.5) -- (6.95,-0.5);
 

\draw (2.85,.8) node[below] {$v_1$};
\draw (7.15,-.3) node[below] {$v_2$};
\draw (2.85,-.3) node[below] {$v'_1$};
\draw (7.15,0.8) node[below]{$v'_2$};
\draw (3.7,0.65) node[below] {$\gamma$};
\draw (6.4,0.73) node[below] {$\gamma'$};
\draw (5,-.45) node[below] {$\beta'$};

\end{tikzpicture}
\end{adjustbox}
\caption{Mesh relations $P_{v'_2}P_{v_1}=P_{v'_1}P_{v_2}$ in $\mathcal{C}_T$} \label{mesh0}
\end{figure}

The following lemma describes the relative positions of diagonals $\gamma$ and $\gamma'$,  when  there exist a nonzero morphism between them.

\begin{lemma}\cite[Lemma 2.1]{schiffler1}\label{Lralf1} The vector space $\Hom_{\mathcal{C}_T}(\gamma,\gamma')$ is nonzero if and only if there exists a diagonal $\tau_i\in T$ such that $\tau_i$ crosses the diagonals $\gamma$ and $\gamma'$ and  the relative positions of them are as in Figure \ref{relativeposition}. That is, let $v_1,v_2$ be the endpoints of $\tau_i$ and $u_1,u_2$ (respectively $u'_1,u'_2$) be the endpoints of $\gamma$ (respectively $\gamma'$). Then ordering the vertices of the polygon in the positive trigonometric direction starting at $v_1$, we have $v_1<u_1\leq u'_1<v_2<u_2\leq u'_2$. In this case, $\Hom_{\mathcal{C}_T}(\gamma,\gamma')$ is of dimension one.
\end{lemma}
\begin{figure}[h]
\begin{adjustbox}{scale=0.8,center}
\begin{tikzpicture}
    
    \draw [] circle (2cm) ;
   \draw [line width = 0.8pt, blue] (-1,1.7) -- (1,-1.7);
   \draw [line width = 0.8pt, blue] (-1.9,0.6) -- (1.9,0.5);

\draw [line width = 0.8pt, blue,dotted] (1,1.7) -- (0.5,-2);
\draw (-1.1,2.1) node[below] {$u_1$};
\draw (1.1,-1.7) node[below] {$u_2$};
\draw (-2.1,1) node[below] {$u'_1$};
\draw (2.14,0.88) node[below]{$u'_2$};
\draw (-0.6 ,1.6) node[below]{$\gamma$};
\draw (-1.4,1.1) node[below]{$\gamma'$};
\draw (1.1,2.1) node[below] {$v_1$};
\draw (0.5,-1.9) node[below] {$v_2$};
\draw (0.76 ,1.6) node[below]{$\tau_i$};
 \end{tikzpicture}
\end{adjustbox}
\caption{Relative position} \label{relativeposition}
\end{figure}
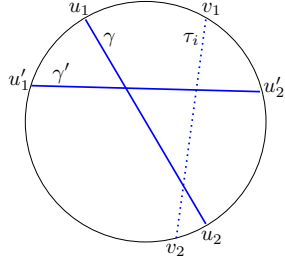

A triangulation $T$ of the $(n+3)$-gon is said to be \textit{triangulation without internal triangles} if each triangle has at least one side on the boundary of the polygon. It is important to recall that every  triangulation $T$ of the $(n+3)$-gon gives rise to a cluster-tilted algebra $kQ_T\slash I$  of type $\mathbb{A}_n$, where $I$ is the two-sided ideal generated by all length two subpath of oriented 3-cycles in $Q_T$, and every cluster-tilted algebra is of this form. In particular, every Dynkin quiver of type $\mathbb{A}_n$ corresponds to a triangulation without internal triangles. The map  associates a  quiver $Q_T$   to the triangulation $T=\lbrace \tau_1,\dots,\tau_n\rbrace$ of $\Pi_{n+3}$ as follows: The vertices of $Q_T$ are $(Q_T)_0= \lbrace1, 2,\dots, n\rbrace$ and there is an arrow $x\rightarrow y$ in $(Q_T)_1$ precisely if the diagonals $\tau_x$ and $\tau_y$ bound a triangle in which $\tau_y$ lies counter-clockwise from $\tau_x$ (see Figure \ref{counterclockwise} and  Example \ref{triangulation1}).\\ 

\begin{figure}[h]

\begin{tikzpicture}
\draw (0,0) 
  -- (2,0) 
  -- (1,1.7)
  -- cycle;
  \node (A) at (2.5,0.5) {$\rightsquigarrow$};
 \node (B) at (3.5,0.5) {$x\rightarrow y$};
 \node (C) at (1.7,0.85) {$_{\tau_x}$};
  \node (D) at (0.25,0.85) {$_{\tau_y}$};
\end{tikzpicture}
\caption{$\tau_y$ is counter-clockwise from $\tau_x$} \label{counterclockwise}
\end{figure}
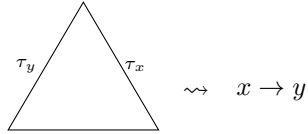

A vertex $x\in(Q_T)_0$ belongs to the \textit{support} $\supp\gamma$ of a diagonal $\gamma\notin T$ if the diagonal $\tau_x \in T$  crosses $\gamma$. The following lemma permits to see diagonals that are not in $T$ as indecomposable objects in the module category of the cluster-tilted algebra $kQ_T\slash I$ of type $\mathbb{A}$. 

\begin{lemma}\cite[Lemma 3.2]{schiffler1}\label{Lralf2} Let $\gamma$ be a diagonal which does not belong to $T$. The set $\supp\gamma$ is connected as a subset of the quiver $Q_T$.
\end{lemma}

In \cite{schiffler1}, the authors defined a $k$-linear additive functor $\Theta$ from $\mathcal{C}_T$ to the category $\md kQ_T\slash I$ of finitely generated $kQ_T\slash I$-modules.  The image of a diagonal $\gamma\notin T$  is the representation $M^{\gamma}=(M_x^{\gamma},f_{\alpha}^{\gamma})$ defined as follows:   For each vertex $x$ in $Q_T$,
    $$M_x^{\gamma}= \left\{ \begin{array}{ll}
k &   \mbox{if } x\in\supp\gamma, 
\\ 0 &  \mbox{otherwise.} 
\end{array}
\right. $$
For any arrow $\alpha:x\rightarrow y$ in $Q_T$, 

$$f_{\alpha}^{\gamma}= \left\{ \begin{array}{lc}
\mbox{id}_k &   \mbox{if }  M_x^{\gamma}= M_y^{\gamma}=k,\\
 0 &  \mbox{otherwise.} 
\end{array}
\right.$$ 

Moreover, for any pivoting elementary move $P:\gamma\to\gamma'$ they defined the morphism $\Theta(P)$ from  $(M_x^{\gamma},f_{\alpha}^{\gamma})$ to  $(M_x^{\gamma'},f_{\alpha}^{\gamma'})$ to be $\mbox{id}_k$ whenever possible and $0$ otherwise. The category $\mathcal{C_T}$ of diagonals  gives a geometric realization of the category of finitely generated $kQ_T\slash I$-modules in the following sense.  
\begin{theorem}\cite[Theorems 4.4 and 5.1]{schiffler1}\label{TRalf}
\begin{itemize}
\item[(a)]The functor $\Theta$  is an equivalence of categories.  
\item[(b)] The irreducible morphisms of $\mathcal{C}_T$ are direct sums of the generating morphisms given by pivoting elementary moves.
\item[(c)]The mesh relations of $\mathcal{C}_T$ are the mesh relations of the AR-quiver of $\mathcal{C}_T$.
\item[(d)]The AR-translation is given on diagonals by $r^{-}$. Here, $r^-$ (respectively $r^+$) denotes the clockwise (respectively counter-clockwise) elementary rotation of the regular polygon $\Pi_{n+3}$.
\item[(e)] The projective indecomposable objects of $\mathcal{C}_T$ are the diagonals in $r^{+}(T)$.
\item[(f)] The injective indecomposable objects of $\mathcal{C}_T$ are  the diagonals in $r^{-}(T)$. \end{itemize}
\end{theorem}

\subsection{Socle-projective modules over incidence algebras}
In this section, we recall some of the main results regarding socle-projective modules over incidence algebras of  posets due to Simson \cites{simson3, simson4}.\\

We denote by $(\P,\preceq)$ a finite  partially ordered set (in short, poset) with respect to the partial order $\preceq$. We shall write $x\prec y$ if $x\preceq y$ and $x\neq y$. For the sake of simplicity we write $\P$ instead of $(\P,\preceq)$. Let $\max\P$ (respectively $\min\P$) be the set of all maximal  (respectively minimal) points of $\P$. A poset $\P$ is called an \textit{r-peak poset} if $\left| \max\P\right|=r$. We recall  that a full subposet $\P'$ of $\P$  is said to be a \textit{peak-subposet} if $\max\P'\subseteq\max\P$. In the sequel, we denote  $\P^{-}=\P\setminus\max\P $. \\

The \textit{Hasse diagram} of $\P$ is obtained as follows: One represents each element of $\P$ as a vertex  in the plane and draws a line segment or curve that goes upward from $x$ to $y$ whenever $y$ \textit{covers} $x$, that is, whenever $x \prec y$ and there is no $z$ such that $x \prec z \prec y$. These lines may cross each other but must not touch any vertices other than their endpoints. Such a diagram, with labeled vertices, uniquely determines its partial order.

\begin{example}\label{onepeak} Given the  one-peak poset $\P$ whose  Hasse diagram is 
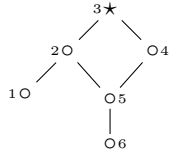
\begin{center}
\begin{tikzpicture}
\node (1) at (0,0) {$\star$};
\node [left  of=1, node distance=0.15cm] (l2)  {$_{_3}$};
\node [below right  of=1, node distance=0.8cm] (2)  {$\circ$};
\node [right  of=2, node distance=0.18cm] (l2)  {$_{_4}$};

\node [below left of=1, node distance=0.8cm] (1')  {$\circ$};
\node [left  of=1', node distance=0.15cm] (l1)  {$_{_2}$};

\node [below  of=1, node distance=1.2cm] (2')  {$\circ$};
\node [right  of=2', node distance=0.18cm] (l4)  {$_{_5}$};

\node [below left of=1', node distance=0.8cm] (1'')  {$\circ$};
\node [left  of=1'', node distance=0.15cm] (l3)  {$_{_1}$};
\node [below  of=2', node distance=0.6cm] (2'')  {$\circ$};
\node [right  of=2'', node distance=0.18cm] (l2)  {$_{_6}$};

\draw [shorten <=-2pt, shorten >=-2pt] (1) -- (1');
\draw [shorten <=-2pt, shorten >=-2pt] (2) -- (2');

\draw [shorten <=-2pt, shorten >=-2pt] (1') -- (1'');
\draw [shorten <=-2pt, shorten >=-2pt] (2') -- (2'');
\draw [shorten <=-2pt, shorten >=-2pt] (1') -- (2');
\draw [shorten <=-2pt, shorten >=-2pt] (1) -- (2);

\end{tikzpicture}
\end{center}

the subposet $\{2,4,5,6\}$ is not a peak-subposet of $\P$, whereas $\{1,6,3\}$ is a peak-subposet of $\P$.
\end{example}

For a point  $a\in\P$, the subposets  of $\P$ 
$$a^{\triangledown}=\left\{{x\in\mathscr{P}\mid a\preceq x}\right\},
\hspace{0.2cm}a_{\vartriangle}=\left\{{x\in\mathscr{P}\mid x\preceq a}\right\}$$
are called  \textit{up-cone and  down-cone} respectively. In the literature, the up-cone of $a$ is also called the \textit{principal filter} of $a$ and its down-cone is its \textit{principal ideal}. A poset $\P$ is called a \textit{chain} (or a \textit{totally
ordered set} or a \textit{linearly ordered set}) if and only if  for
all $x,y\in \P$ we have $x\preceq y$ or $y\preceq x$.  On the other hand, an ordered set $\P$ is called an \textit{antichain} if and only if for all $x,y\in\P$ we have $x\preceq y$ in $\P$ only if $x=y$.  The cardinality of a maximal  antichain in a poset $\mathscr{P}$ is called the \textit{width} $w(\P)$ of $\P$. If some subsets $X_1,\dots, X_n$ of  $\P$ do not intersect mutually (but may have comparable points), then their union $X_1\cup\dots\cup X_n$  is called a \textit{sum} and is denoted by $X_1+\dots+X_n$. We recall that according to the Dilworth's theorem if the width $w(\P)=n$ then $\P$ is a sum of $n$ chains. For details on posets, we refer to  \cite{Priestley,stanley}.\\

Given a finite poset $\P$,  by $k\P$ we mean the \textit{incidence algebra} of the poset $\P$. $k\P$ can be described as a bound quiver algebra $kQ/I$ induced by the \textit{Hasse quiver} $Q$ of $\P$ whose vertices are the points of $\mathscr{P}$ and there is an arrow $\alpha:x\rightarrow y$ for each pair $x,y\in\P$ such that $y$ covers $x$.  The ideal $I$ is generated by all the commutativity relations $\gamma-\gamma'$  with  $\gamma$ and $\gamma'$ parallel paths in $Q$. In this case, the category $\md(k\P)$ of the finitely generated $k\P$-modules is identified with the well known category $\rep (Q,I)$ of representations  of the bound quiver $(Q,I)$.

\begin{example}\label{Hassequiver} The incidence algebra $k\P$ of the poset $\P$ defined in Example \ref{onepeak} is the bound quiver algebra $kQ\slash I$, where  $Q$ is the  is the quiver 
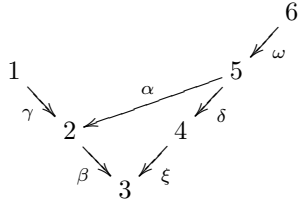
\[\xymatrix@=1em{&&&&&6\ar[ld]^\omega\\
1 \ar[rd]_\gamma &&&& 5\ar[llld]_\alpha\ar[ld]^\delta\\
& 2 \ar[rd]_\beta & & 4\ar[ld]^\xi\\
&&3}\] and the ideal $I$ is generated by the relation $\alpha\beta-\delta\xi$. 
\end{example}

Recall that the \textit{socle} $\text{soc }M$ of a module $M$ is the semisimple submodule generated by  all simple submodules of $M$. A module $M$ is called \textit{socle-projective} if $\text{soc }M$ is a projective module. We denote by $\md_{sp}(k\P)$ the full subcategory of $\md (k\P)$  whose objects are the socle-projective $k\P$-modules. We have an explicit description of the objects in $\md_{sp}(k\P)$ as follows.

\begin{proposition}\cite[Section 3]{simson5}\label{sp module} Each $k\P$-module $M$ in $\md(k\mathscr{P})$ is identified with  a collection $M=(M_x,{_y}h_x)_{x,y\in\P}$  of finite-dimensional $k$-vector spaces $M_x$, one for each point $x\in\P$, and a collection of $k$-linear maps ${_y}h_x: M_x \to M_y$, one for each relation $x\preceq y$ in $\P$, such that

\begin{itemize}
\item[(a)]${_x}h_x$ is the identity of $M_x$  for all $x\in\P$ and ${_w}h_y\cdot {_y}h_x= {_w}h_x$ for all $x\preceq y\preceq w$ in $\P$.
\end{itemize}
Furthermore, $M=(M_x,{_y}h_x)_{x,y\in\P}$ is a socle-projective module if  it also holds that  
\begin{itemize}
\item[(b)] $\stackbin[z\in\max\P]{}{\bigcap}\ker {_z}h_x=0$ for all $x\in\P^{-}$ such that $x\prec z$.
\end{itemize}
\end{proposition} 
Note that it is enough to define the linear maps ${_y}h_x$ when $y$ covers $x$, that is,  one for each arrow in the Hasse quiver of $\P$ because if $x\prec y$ but $y$ does not cover $x$ then for any chain $x=x_0\prec x_1\prec\dots \prec x_{l}=y $ in $\P$ such that $x_{i+1}$ covers $x_i$ we have that ${_y}h_x={_y}h_{x_{l-1}}\cdots {_{x_{1}}}h_x$. The condition (a) implies that it is well defined. 

\begin{example}\label{spmodule} Let $\P$ be the one-peak poset given in Example \ref{onepeak} whose Hasse quiver is shown in Example \ref{Hassequiver}.   The $k\P$-module $M$ given by the system
\[\xymatrix@=1.3em{&&&&& k \ar[ld]^{\left(\begin{smallmatrix}1\\0\end{smallmatrix}\right)}\\
k \ar[rd]_{\left(\begin{smallmatrix}1\\0\end{smallmatrix}\right)} &&&& k^2\ar[llld]_1\ar[ld]^1\\
& k^2 \ar[rd]_1 & & k^2 \ar[ld]^1\\
&&k^2}\]
is socle-projective. Indeed, since $\P$ has a unique maximal point $z=3$ we have that   $\stackbin[z\in\max\P]{}{\bigcap}\ker {_z}h_x=\ker {_3}h_x$ for all $x\in \P^-$. 	Note that, the kernel $\ker {_3}h_x$ of the map  ${_3}h_x$ from $M_x$ to $M_3$ is zero for each $x=1,2,4,5,6$.\\

For example, replacing $\left(\begin{smallmatrix}1\\0\end{smallmatrix}\right)$ on arrow $\gamma$  by $\left(\begin{smallmatrix}0\\0\end{smallmatrix}\right)$ would resulting a representation that is not socle-projective.
\end{example}

Let $M=(M_x,{_y}h_x)_{x,y\in\P}$ and $N=(N_x,{_y}h'_x)_{x,y\in\P}$   be two objects in $\md(k\P)$. A \textit{morphism}  of $k\P$-modules $f:M\to N$ is a collection $f=(f_x)_{x\in\P}$ of linear maps
$$f_x:M_x\to N_x$$ 
such that for each relation $x\preceq y$ in $\P$ the diagram 

\begin{center}
\begin{tikzcd}[row sep=large, column sep = large]
M_x \arrow[r, "{_y}h_x"] \arrow[d,"f_x" ]
& M_y \arrow[d, "f_y" ] \\
N_x \arrow[r, "{_y}h'_x" ]
& N_y
\end{tikzcd}
\end{center}
commutes, that is, $$f_y\circ{_y}h_x={_y}h'_x\circ f_x.$$

On the other hand,  from the point of view of the representations of posets introduced by Nazarova and Roiter \cite{nazarova}, the category $\md_{sp}(k\mathscr{P})$ is identified with the category $\P$-spr of \textit{peak $\P$-spaces} (or \textit{socle-projective representations of $\P$}) over the field $k$ defined by Simson \cite{simson3}. The objects of $\P$-spr are systems $$M=(M_x)_{x\in\P}$$ of finite dimensional $k$-vector spaces $M_x$ satisfying the following conditions:

\begin{enumerate}
\item[(a)] For each $x\in\P$, $M_x$ is a $k$-subspace of the space $$M^{\bullet}=\bigoplus_{z\in\max\P}M_{z}.$$

 \item[(b)] The inclusion $M_{z}\hookrightarrow M^{\bullet}$ is defined as usual for each $z\in\max\P$.

\item[(c)] For each  $x\prec y$ in $\P$ it holds that $\pi_{y}(M_{x})\subseteq M_{y}$, where $\pi_{y}\in\mathrm{End}\hspace{0.1cm}M^{\bullet}$ is the composition of the direct summand projection of $M^{\bullet}$ on  $$M^{\bullet}_y=\bigoplus_{y\preceq z\in\max\P}M_z$$ with the natural embedding homomorphism $M^{\bullet}_y\hookrightarrow M^{\bullet}$. 
\item[(d)] If $z\in\max\P$ and  $x\notin z_{\vartriangle}$ then $\pi_z(M_x)=0$.
\end{enumerate}

A \textit{morphism} $f:M\to L$ between two peak $\P$-spaces $M$ and $L$ is a collection of $k$-linear maps $f=(f_z:M_z\to L_z)_{z\in\max\P}$ such that for all $x\in\P$
$$\Big(\bigoplus_{z\in\max\P}f_z\Big)(M_x)\subseteq L_x .$$ 

The category of peak $\P$-spaces over $k$ is denoted by  $\repr$.  Following \cites{simson3,simson4, pena}, $\repr$  is an additive Krull–Schmidt category of finite global dimension which  is closed under taking kernels and  extensions. Furthermore,  it  has  enough projective objects, AR-sequences, source maps, and sink maps.\\

The \textit{coordinate vector} of a peak $\P$-space $M$ is the vector $$d = \dm M = (d_x)_{x\in\P}\in\mathbb{Z}^{\P}$$ such that $d_x= \dim_{k} M_x$ if $x\in\max\P$ and  $d_x=\dim_{k}(M_x\slash \underline{M_x})$ otherwise, where $\underline{M_x} =\sum_{y\prec x}\pi_x(M_y)$. This vector allows us to define the \textit{coordinate support} $\csupp M$ of  $M$ given by the peak-subposet $\csupp M =\lbrace x\in\P\hspace{0.1cm}\vert\hspace{0.1cm} (\dm M)_x\neq 0\rbrace$ of $\P$. 

\begin{example}\label{morphism} Let $\P$ be the one-peak poset given in Example \ref{onepeak}. The system $M=(M_1,M_2,M_3,M_4,M_5,M_6)=(k\oplus 0,k^2,k^2,k^2,k^2,k\oplus 0)$ is a peak $\P$-space, whereas the system $(0,0,k,k,k,k)$ does not satisfy condition (c) because $ M_5$  is not a $k$-subspace of $M_2$. In this case, the coordinate vector is  $\dm M=(1,1,2,0,1,1)$ and the coordinate support is the peak-subposet  $\csupp M=\{1,2,3,5,6\}$  of $\P$. For greater clarity, note that $\underline{M_2}=\pi_2(M_1)=M_1=k\oplus 0$, where $\pi_2$ is the identity of $k^{2}$, that is,  $\dim_{k}(M_2\slash  \underline{M_2})=1$. On the other hand, given the peak $\P$-space $ L=~(k,k,k,k,k,k)$,  the map $
f:k^2\xrightarrow{\left(\begin{smallmatrix}1&1\end{smallmatrix}\right)}k$ is a morphism from $M$ to $L$.
\end{example}

A peak $\P$-space $M$  is said to be \textit{sincere} if it is indecomposable and $\csupp M=\P$. Furthermore, if there exists a sincere peak $\P$-space we say that $\P$ is a \textit{sincere poset}. \\

The category $\repr$ (or the poset $\P$) is said to be of \textit{ finite representation type} if it has  only a finite number of nonisomorphic indecomposable peak $\P$-spaces, otherwise, it is of  \textit{infinite representation type}. The classification of all sincere  $r$-peak posets of representation finite type  was given by M. Kleiner \cite{kleiner75}, for the case $r=1$, and by J. Kosakowska \cite{justina,justina1,justina2}, in the remaining cases. Moreover, they gave the list of the sincere peak $\P$-spaces, where $\P$ is a sincere poset of finite representation type. Such lists are important because we can get all indecomposable peak $\P$-spaces of a  given poset $\P$  lifting all sincere $\mathcal{S}$-spaces of all sincere peak-subposets $\mathcal{S}$ of $\P$ via the well-known \textit{subposet induced functor} \cite{dowbor} 
\begin{equation}\label{inducedfunctor}
{\mathcal{S}}:\mathcal{S}\mbox{-spr}\to\repr
\end{equation}
that assigns to the peak $\mathcal{S}$-space $(M_x)_{x\in\mathcal{S}}$  the peak $\P$-space $(\hat{M}_x)_{x\in\P}$,  where $\hat{M}_x$ is defined by    
\[\hat{M}_{x}=
\begin{cases}
M_x,  & \text{if }x\in \max\mathcal{S}, \\
0, & \text{if }x\notin (\max\mathcal{S})_{\vartriangle},\\
\sum_{y\preceq x} \pi_x(M_y),                                             &\text{if }x\in (\max\mathcal{S})_{\vartriangle}.
\end{cases}\]

\begin{proposition}\cite[Proposition 5.14]{simson} \label{ind} Up to isomorphism, any indecomposable object $M$ in $\repr$ is   the image $T_{\mathcal{S}}(L)$ of a sincere peak $\mathcal{S}$-space $L\in\mathcal{S}$-spr, where $\mathcal{S}$ is a sincere peak-subposet of $\P$. As a consequence, 
$\csupp M=\mathcal{S}$.  
\end{proposition}

For a finite poset $\P$, two functors play a important role in the proof of the equivalence of the categories $\repr$ and $\md_{sp}(k\P)$; one of them is the called  \textit{embedding functor} \begin{equation}\label{embedding}
\bm\rho: \repr\to\md(k\P)
\end{equation}
defined by $\bm\rho({U})=(U_y,{_y}\pi_x)_{x,y\in\P}$, where ${_y}\pi_x:U_x\to U_y$ is the unique $k$-linear map making the diagram
\begin{center}
\begin{tikzcd}[row sep=large, column sep = large]
U_x \arrow[r, hook] \arrow[d,"{_y}\pi_x" ]
& U^{\bullet} \arrow[d, "\pi_y" ] \\
U_y \arrow[r, hook]
& U^{\bullet}
\end{tikzcd}
\end{center}
commutative. The other functor is the called  \textit{adjustment functor} \cites{simson, simson3} \begin{equation}\label{adjustment}
\bm\theta:\md(\Bbbk\P)\to\repr
\end{equation} given by $\bm\theta(M_x,{_y}h_x)_{x\text{, }y\in\P}=(U_x)_{x\in\P}$, such that

 $$U_x= \left\{ \begin{array}{ll}
M_x &   \mbox{if } x\in\max\P,
\\ \mbox{Im}(f:M_x\to \bigoplus_{z\in\max\P}M_z)  &  \mbox{otherwise.}
\end{array}
\right.
$$
 where $f=({_z}f_x)_{z\in\max\P}$ and $${_z}f_x= \left\{ \begin{array}{ll}
{_z}h_x &   \mbox{if } x\prec z \mbox{ in }\P ,
\\ 0  &  \mbox{otherwise.}
\end{array}
\right.
$$
\begin{example} The peak $\P$-space $M$ given in Example  \ref{morphism} is the image  by $\bm \theta$  of the socle-projective $k\P$-module  defined in Example \ref{spmodule}. Moreover, the functor $\bm\rho$ sends the  peak $\P$-space $M$  to the  socle-projective representation mentioned. 
\end{example}

Then, we have the following equivalence of categories given by Simson.

\begin{lemma}\cite[Lemma 2.1]{simson3}\label{prin}
The functor $\bm\rho$ induces an equivalence of categories $\bm\rho: \repr\to\md_{sp}(k\mathscr{P})$ and the inverse of $\bm\rho$ is the restriction of $\bm\theta$ to $\md_{sp}(k\mathscr{P})$.\end{lemma}

Recall that a module $M$ in $\md (k\P)$ is called \textit{prinjective} (not to be confused with preinjective) if it has a  projective resolution of the form $$0\to P'\to P\to M\to 0$$ in $\md(k\mathscr{P})$, where $P$ is  projective  and $P'$ is semisimple projective.  The full subcategory $\prin(k\P)$ of $\md(k\P)$ generated by the prinjective modules is closely related to the category $\repr$ because the functor $\bm\theta$ induces a full dense additive  functor
$$\bm\theta_{\P}:\prin(k\P)\to \repr$$ such that $\ker\bm\theta_{\P}$ consist of all maps in $\prin(k\P)$ having a factorization through a direct sum of copies of the projective $k\P^{-}$-modules. Furthermore, the functor $\theta_{\P}$ preserves and reflects finite representation type, and induces an equivalence of categories $$\prin(k\P)\slash\ker\bm\theta_{\P} \cong \repr.$$

Thus, we can formulate the following criterion of finite representation type due Simson.
\begin{theorem}\cite[Theorem 3.1]{simson3}\label{tipof} The following conditions are equivalent
\begin{itemize}
\item[(a)] The category $\repr$ is of finite representation type.
\item[(b)] The category $\prin(k\P)$ is of finite representation type.
\item[(c)] The poset $\P$ does not contain as a 
peak-subposet  any of the posets $\P_1,\dots,\P_{110}$  presented in \cite[Section 5]{simson3}.
\end{itemize}
\end{theorem}

\section{Posets of type $\mathbb{A}$}\label{posetstypeA}
In this section, we introduce a family of posets which we call posets of type $\mathbb{A}$ because of a characterization using a type $\mathbb{A}$ quiver given in Proposition \ref{0}.

\begin{definition}\label{defposetypeA}
A finite connected poset $\P$ is said to be \textit{poset of type $\mathbb{A}$} if $\P$ does not contain as a peak-subposet any of the following posets:\vspace{0.2cm}
\begin{center}

\begin{tabular}{|l|l|l|l|}
\hline 
\begin{tikzpicture}
\node (top) at (-0.7,0.5) {$\mathcal{R}_1$};

\node (top) at (0,0) {$\star$};
\node [below of=top, node distance=0.7cm](center) {$\circ$};
\node [left  of=center, node distance=0.7cm] (left)  {$\circ$};
\node [right of=center, node distance=0.7cm] (right) {$\circ$};

    \draw [blue,  thick, shorten <=-2pt, shorten >=-2pt] (top) -- (left);
    \draw [blue, thick, shorten <=-2pt, shorten >=-2pt] (top) -- (right);
    \draw [blue, thick, shorten <=-2pt, shorten >=-2pt] (top) -- (center);\end{tikzpicture} &
    
    \begin{tikzpicture} 
    \node (top) at (0,0.5) {$\mathcal{R}_2$};
    \node (top1) at (0,0) {$\star$};
    \node [right of=top1, node distance=0.7cm](top2)  {$\star$};
    \node (m) at (0.35,-0.35) {$\circ$};
    \node (bellow) at (0.35,-0.7) {$\circ$};
     \draw [blue,  thick, shorten <=-2pt, shorten >=-2pt] (top1) -- (m);
    \draw [blue, thick, shorten <=-2pt, shorten >=-2pt] (top2) -- (m);
    \draw [blue, thick, shorten <=-2pt, shorten >=-2pt] (m) -- (bellow);
    
   \end{tikzpicture}&
   
    \begin{tikzpicture} 
\node (top) at (0,0.5) {$\mathcal{R}_3$};   
   \node (top3) at (0,0){$\star$};    
    \node [right of=top3, node distance=0.7cm](top4) {$\star$};
        \node [right of=top4, node distance=0.7cm] (top5) {$\star$};
        \node [below of=top4, node distance=0.7cm] (be) {$\circ$};
 \draw [blue, thick, shorten <=-2pt, shorten >=-2pt] (be) -- (top3);
 \draw [blue, thick, shorten <=-2pt, shorten >=-2pt] (be) -- (top4);
 \draw [blue, thick, shorten <=-2pt, shorten >=-2pt] (be) -- (top5);\end{tikzpicture} & 
 
 \begin{tikzpicture}
 \node (1) at (0,0) {$\star_1$};
\node (A) at (0.8,0.5) {$\mathcal{R}_{4,n}$, $n\geq 0$};
\node [right of=1, node distance=0.7cm](2) {$\star_2$};

\node [right of=2, node distance=0.7cm](3) {$\star_3$};
\node [right of=3, node distance=0.7cm] (4)  {$\cdots$};
\node [right of=4, node distance=0.7cm] (5) {$\star_{n+2}$};

\node [below of=1, node distance=0.7cm] (1')  {$\circ$};
\node [below  of=2, node distance=0.7cm] (2')  {$\circ$};
\node [below  of=3, node distance=0.7cm] (3')  {$\circ$};
\node [below  of=4, node distance=0.7cm] (4')  {$\dots$};
\node [below  of=5, node distance=0.7cm] (5')  {$\circ$};

   \draw [blue,  thick, shorten <=-2pt, shorten >=-2pt] (1) -- (1');
    \draw [blue, thick, shorten <=-2pt, shorten >=-2pt] (2') -- (2);
    
\draw [blue,  thick, shorten <=-2pt, shorten >=-2pt] (3) -- (3');
    \draw [blue, thick, shorten <=-2pt, shorten >=-2pt] (5) -- (5');    
    
    \draw [blue, thick, shorten <=-2pt, shorten >=-2pt] (1) -- (2');
    
 \draw [blue, thick, shorten <=-2pt, shorten >=-2pt] (2) -- (3');
     \draw [blue, thick, shorten <=-2pt, shorten >=-2pt] (1') -- (5);
\end{tikzpicture} \\ 
\hline 
\end{tabular} 

\end{center}
\vspace{0.2cm}
\end{definition}

\begin{example}\label{posetstypeA} The poset $\P$ given  in Example \ref{onepeak} is a poset of type $\mathbb{A}$ because although $\{2,4,5,6\}$ is a subposet  of type $\mathcal{R}_2$ it is not a peak-subposet of $\P$. On the other hand, the poset
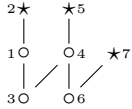
\begin{center}
\begin{tikzpicture}
\node (1) at (0,0) {$\star$};
\node [left  of=1, node distance=0.15cm] (l2)  {$_{_2}$};
\node [right  of=1, node distance=0.6cm] (2)  {$\star$};
\node [right  of=2, node distance=0.18cm] (l2)  {$_{_5}$};

\node [below  of=1, node distance=0.6cm] (1')  {$\circ$};
\node [left  of=1', node distance=0.15cm] (l1)  {$_{_1}$};

\node [below  of=2, node distance=0.6cm] (2')  {$\circ$};
\node [right  of=2', node distance=0.18cm] (l4)  {$_{_4}$};

\node [ below  of=1', node distance=0.6cm] (1'')  {$\circ$};
\node [left  of=1'', node distance=0.15cm] (l3)  {$_{_3}$};
\node [below  of=2', node distance=0.6cm] (2'')  {$\circ$};
\node [right  of=2'', node distance=0.18cm] (l2)  {$_{_6}$};
\node [right  of=2', node distance=0.6cm] (3)  {$\star$};
\node [right  of=3, node distance=0.18cm] (l2)  {$_{_7}$};

\draw [shorten <=-2pt, shorten >=-2pt] (1) -- (1');
\draw [shorten <=-2pt, shorten >=-2pt] (2) -- (2');

\draw [shorten <=-2pt, shorten >=-2pt] (1') -- (1'');
\draw [shorten <=-2pt, shorten >=-2pt] (2') -- (2'');
\draw [shorten <=-2pt, shorten >=-2pt] (1'') -- (2');
\draw [shorten <=-2pt, shorten >=-2pt] (2'') -- (3);
\end{tikzpicture}
\end{center}
is  a three-peak  poset of type $\mathbb{A}$ which can be viewed as a Dynkin quiver of type $\mathbb{E}_7$.\end{example}

We say that two maximal points $z$ and $z'$ in a poset $\P$  are \textit{neighbors} if $z_{\vartriangle}\cap z'_{\vartriangle}\neq \emptyset$.  Then, we describe this notion when $\P$ is a poset of type $\mathbb{A}$ as follows. 

\begin{lemma}\label{neighbor}
Let $\P$ be  an $r$-peak  poset of type $\mathbb{A}$ with $r\geq 2$. The following statements hold:
\begin{itemize}
\item[(a)]The points $z,z'\in\max\P$ are neighbors if and only if $z_{\vartriangle}\cap z'_{\vartriangle}=\lbrace x\rbrace$, for some $x\in\min\P$.
\item[(b)]There exists a point $z\in\max\P$ such that $z$ has a unique neighbor.
\end{itemize}
\end{lemma}
\begin{proof}
Since $\mathcal{R}_2$ is not  peak-subposet of $\P$ then  $x\in z_{\vartriangle}\cap z'_{\vartriangle}$ implies that $x\in\min\P$ because otherwise there exists $y\prec x$ and then the subposet $\{y,x,z,z'\}$ is of the form $\mathcal{R}_2$. Now, if $x\neq x'\in z_{\vartriangle}\cap z'_{\vartriangle} $ then $\mathcal{R}_{4,0}$ is peak-subposet of $\P$ which is a contradiction. Thus, the set $z_{\vartriangle}\cap z'_{\vartriangle}$ is a singleton.   Clearly the converse implication is true. On other hand, since $\P$ is a connected poset then each maximal  point $z$ has  at least one neighbor, but $z$ does not have three neighbor points. Indeed, if $z_1,z_2,z_3$ are distinct neighbors of $z$  with $x_i\in z_\vartriangle\cap (z_i)_{\vartriangle}$ then we have the subposet 
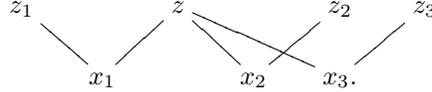
\[\xymatrix@=1.5em{z_1\ar@{-}[rd]&&z\ar@{-}[rd]\ar@{-}[ld]&&z_2\ar@{-}[ld] &z_3\ar@{-}[ld]\\
&x_1&&x_2&x_3.\ar@{-}[llu]
}\]
If $x_1,x_2,x_3$ are three distinct points then $\{z,x_1,x_2,x_3\}$ is a peak-subposet of type $\mathcal{R}_1$, a contradiction. Suppose that two of the $x_i$ are equal, for instance $x_1=x_2$. Then $\{z_1,z,z_2,x_1\}$ is a peak-subposet of type $\mathcal{R}_3$, a contradiction. Thus $z$ has at most two neighbors. Finally, if each maximal point has exactly two neighbor points then $\mathcal{R}_{4,n}$  is peak-subposet of $\P$ for some $n\geq 0$, which is contradictory, and  we are done.
\end{proof}

Actually, the posets of  type $\mathbb{A}$ can be viewed as posets associated to certain quivers which are obtained from Dynkin quivers of type $\mathbb{A}$  by adding some new arrows. To explain this, we need the following definitions:\\

Let $Q$ be an acyclic quiver and let $\P_Q=Q_0$ be its set of vertices. We define an order on $\P_Q$ by $x\preceq y$ if and only if there exists a path from $x$ to $y$ in $Q$. We say that $\P_Q$ is the \textit{poset associated to the quiver} $Q$. Note that there is a unique poset associated to a finite acyclic quiver, but the converse is false in general.
As an example, the poset  associated to  the quiver  
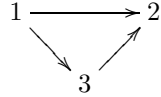
\[\xymatrix@=1.5em{1\ar[rr]\ar[rd]&&2\\&3\ar[ru]}\] is  $\{1<3<2\}$. However, the Hasse quiver of this poset is $\xymatrix@=1.5em{1\ar[r]&3\ar[r]&2}$. Thus, the two quivers have the same associated poset.  As another example, corresponding to the poset $\P=\lbrace 1,2\rbrace$ together with the usual ordering $1< 2$, we get countably many quivers with $n$ arrows from $1$ to $2$ for any natural number $n\in\mathbb{N}$.\\

Recall that a vertex $x\in Q_0$ is said to be a \textit{sink vertex }  (respectively \textit{source vertex}) if there is no arrow $\alpha$ in $Q_1$
such that  $s(\alpha)=x$ (respectively $t(\alpha)=x$), where 
$s(\alpha)$ is the starting vertex and $t(\alpha)$ is the target vertex of the arrow $\alpha$.\\

Let $Q$ be a Dynkin quiver of type $\mathbb{A}$ and let $z\in Q_{0}$ be a sink vertex. The maximal full subquiver $Q^{(z)}$ of $Q$ with $z$ as the unique sink is called the \textit{$z$-subquiver} of $Q$. In other words, the vertices of  $Q^{(z)}$ are the vertices in  the support $\text{Supp }I(z)$ of the indecomposable injective representation $I(z)$ at vertex $z$. 

\begin{example} \label{zsubquiver} The quiver 
$Q=\xymatrix@=1.5em{1\ar[r]&2&3\ar[l]\ar[r]&4\ar[r]&5&6\ar[l]\ar[r]&7}$ of type $\mathbb{A}_7$  contains the $2$-subquiver 
$\xymatrix@=1.5em{1\ar[r]&2&3\ar[l]}$, 
the $5$-subquiver 
$\xymatrix@=1.5em{3\ar[r]&4\ar[r]&5&6\ar[l]}$, and the $7$-subquiver $\xymatrix@=1.5em{6\ar[r]&7}$.
\end{example}

We will now add new arrows to our quiver $Q$ as follows.  
\begin{definition}\label{aliensetDef}
A set $F=\{\alpha_1,\dots,\alpha_t\}$ of new arrows for $Q$ is called an \textit{alien set} for $Q$ if the following conditions  hold.

\begin{itemize}
\item[(a)] For each $\alpha\in F$ there exists a sink vertex $z$ in $Q$ such that $s(\alpha),t(\alpha)\in \text{Supp }I(z)$.

\item[(b)] $t(\alpha)$  is not a source vertex in $Q$ unless it is an extremal vertex in $Q$.

\item [(c)]  For all $\alpha\in F$,  the  arrow $\alpha$ is the unique path from $s(\alpha)$ to  $t(\alpha)$ in $Q^F$, where $Q^F$  is  the quiver such that $Q^F_0=Q_0$ and $Q^F_1=Q_1\cup F$. 
 
\item[(d)] The quiver $Q^F$ is acyclic.
\end{itemize}
The arrows in an alien set for $Q$ will be called \textit{alien arrows}.
\end{definition}

\begin{example}\label{alienset1}
 If $Q$ is the quiver $\xymatrix@=1.5em{1\ar[r]&2\ar[r]&3&4\ar[l]&5\ar[l]&6\ar[l]}$ of type $\mathbb{A}_6$
then  $F=\{\ \alpha:5\rightarrow2\}$ is an alien set for $Q$ and $Q^F$ is the quiver in Example \ref{Hassequiver}. Note that the poset $\P_{Q^F}$ is the one-peak poset of type $\mathbb{A}$ defined in Example \ref{onepeak}.
\end{example}

\begin{example}\label{alien set}Let $Q$ be the quiver in Example \ref{zsubquiver}. The set $$F=\lbrace\alpha:3\rightarrow 1,\hspace{0.1cm}\beta:6\rightarrow 4\rbrace$$   is an  alien set  for $Q$. Moreover,  the quiver $Q^F$ is equal to

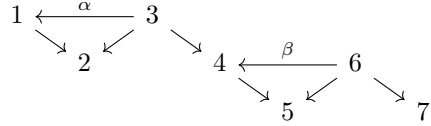
\begin{center}
\begin{tikzcd}[row sep= tiny, column sep = small]
1\arrow[dr]&&3\arrow[dl]\arrow[ll, "\alpha",swap]\arrow[dr]&&&&\\
&2&&4\arrow[dr]&&6\arrow[ll,"\beta",swap]\arrow[dl]\arrow[dr]&\\
&&&&5&&7\\ 
\end{tikzcd}
\end{center}
 Note that the poset $\P_{Q^F}$ associated to  $Q^{F}$ is the three-peak poset of type $\mathbb{A}$ defined in Example \ref{posetstypeA}.
\end{example}

The following  proposition characterizes posets of  type $\mathbb{A}$.

\begin{proposition} \label{0}

A poset $\P$ is of type $\mathbb{A}$ if and only if there exists a Dynkin quiver $Q$ of type $\mathbb{A}$ and an  alien set $F$ for $Q$ such that $\P=\P_{Q^F}$ is the  poset associated to the quiver $Q^F$.
\end{proposition}

\begin{proof}
In order to prove the necessary condition we proceed by induction on the number $r$ of peaks in $\P$. First we suppose that $\P$ is a one-peak poset with a maximal point $z$. Since $\mathcal{R}_1$ is not peak-subposet of $\P$  we conclude that $w(\P)\leq2$. Thus, if $w(\P)=1$ then  $\P$ is a chain and it can be viewed as a  linearly oriented quiver $Q$ of type $\mathbb{A}$. Clearly, if $F=\emptyset$ then $\P$ is the  poset associated  to the quiver $Q^{F}$.  On the other hand, if $w(\P)=2$ then by Dilworth's theorem $\P^{-}$ is a sum of two chains $\P_1=\lbrace x_1 \prec \dots \prec x_s \rbrace$ and $P_2=\lbrace y_1\prec\dots\prec y_t \rbrace$. Given the quiver 

$$Q=\xymatrix@=1.5em{x_1\ar[r]&\cdots\ar[r]&x_s\ar[r]&z&y_t\ar[l]&\cdots\ar[l]&y_1\ar[l]},$$
the set $F=F_1\cup F_2$ such that $F_1=\lbrace\alpha: x\rightarrow y\hspace{0.1cm} \vert\hspace{0.1cm} y\in\P_2 \mbox{ covers } x\in\P_1 \rbrace$ and $F_2=\lbrace\alpha: y\rightarrow x \hspace{0.1cm} \vert\hspace{0.1cm} x\in\P_1 \mbox{ covers } y\in\P_2 \rbrace$ is an alien set for $Q$. Let $\alpha:x\rightarrow y$ be an alien arrow in $F$. We suppose that there is another path from $x$ to $y$ in $Q^{F}$, then there exists an alien arrow  $\alpha':x'\rightarrow y'$ in $Q^{F}$ such that $x\preceq x'$, $y'\preceq y$ and $x\neq x'$ or $y\neq y'$. However, in this case, $y$ does not cover $x$ which is a contradiction. Thus, $F$ is an  alien set for $Q$ and $\P$ is the  poset $\P_{Q^F}$ associated   to the quiver $Q^F$.\\

Now, we suppose that the assertion is true for any $h$-peak poset of type $\mathbb{A}$, for all $1\leq h\leq r-1$. Let $\P$ be  a 
$r$-peak poset of type $\mathbb{A}$. By Lemma \ref{neighbor} part (b)  we can choose a point $z\in\max\P$ such that $z$ has  a unique   neighbor.  The peak-subposets  $\bar{\P}=\lbrace z_1,\dots,z_{r-1} \rbrace_{\vartriangle}$ and $\P_z=z_{\vartriangle}$ of $\P$ are two  posets of type $\mathbb{A}$, where $\max\P=\lbrace z_1,\dots,z_{r-1},z\rbrace$. By induction there are two Dynkin quivers $Q'$ and $Q''$  of type $\mathbb{A}$ and two alien sets $F'$ and $F''$ for  $Q'$ and $Q''$ respectively  such that   $\bar{\P}$ is the  poset associated to the quiver  $Q'^{F'}$ and $\P_z$ is the poset associated to the quiver $Q''^{F''}$.  We suppose that $z'\in(\max\P)\setminus\lbrace z\rbrace$ is the neighbor of the point  $z$.   By Lemma \ref{neighbor} part (a) we conclude that  $z_{\vartriangle}\cap z'_{\vartriangle}=\lbrace x\rbrace$, where $x\in\min\P$, in other words, $x$ is source vertex in $Q'$ and $Q''$. Clearly $\bar{\P}\cap \P_z=\lbrace x\rbrace$, otherwise $z$ would have two neighbors. Now we are going to  prove that the point $x$ is an extremal vertex of both quivers $Q'$ and $Q''$. Since $Q''$ has a unique sink vertex $z$ and $x$ is a source vertex in $Q''$ then $x$ is an extremal vertex in $Q''$. Moreover, if $x$ is a source vertex which is not an extremal vertex in $Q'$ then $\mathcal{R}_3$ would be a peak-subposet of $\P$ and in this way we get a contradiction. Then the quiver $Q=(Q_0,Q_1)$ such that $Q_0=Q'_0\cup Q''_0$ and $Q_1=Q'_1\cup Q''_1$ is a Dynkin quiver of type $\mathbb{A}$. Note also that $F=F'\cup F''$ is an alien set for $Q$ because there is no alien arrow ending at $x$, otherwise $\mathcal{R}_2$ would be a peak-subposet of $\P$. Furthermore, $\P$ is the poset associated to the quiver $Q^{F}$.\\

The sufficiency of the assertion is proved as follows; let us suppose that $\P$ is the  poset $\P_{Q^F}$ associated  to a quiver $Q^{F}$, where $Q$ is a Dynkin quiver of type $\mathbb{A}$ and $F$ is an alien set for $Q$, we shall prove that $\P$ is of  type $\mathbb{A}$.  Locally an alien arrow $\alpha\in F$ with $s(\alpha), t(\alpha)\in \text{Supp } I(z)$, where $z$ is a sink vertex in $Q$ is such that $s(\alpha)\neq z$, otherwise the quiver $Q^F$ would be cyclic. Then the maximal points in $\P$ are exactly the sink vertices in the quiver $Q$. Since the subposet $z_{\vartriangle}=Q^{(z)}_{0}$ of $\P$ is a poset of width at most two then $\mathcal{R}_1$ is not a peak-subposet of $\P$.  On the other hand, by Lemma \ref{neighbor} part (b), if $z,z'\in\max\P$ are neighbors and $x\in z_{\vartriangle}\cap z'_{\vartriangle}$ then  $x\in\min\P$, thus $x\in\min \P_Q$. Since $Q$ is a Dynkin quiver of type $\mathbb{A}$, then $x$ is a source vertex in $Q$. However, $x$ is a non extremal vertex in $Q$ because  an alien arrow  always connects two vertices in the same $z$-subquiver. Definition \ref{aliensetDef} part (b) implies that  $\P$ does not contain $\mathcal{R}_2$ as peak-subposet. Now, we suppose that $\P$ contains $\mathcal{R}_3$ as peak-subposet, that is,  there are three maximal points $z,z',z''$ in $\P$ and a  point $x\in\P$  such that $x\in z_{\vartriangle}\cap z'_{\vartriangle}\cap z''_{\vartriangle}$. Thus, by the same  arguments as above, $z,z'$ and $z''$ are sink vertices in $Q$. Moreover, since $\mathcal{R}_2$ is not a peak-subposet of $\P$, then $x$ is a minimal point in $\P$ which implies that $x$ is  a source vertex in the quiver $Q$.  Moreover, since $Q$ is a Dynkin quiver of type $\mathbb{A}$, we can suppose that there is no path in $Q$ from $x$ to $z''$; thus,   Definition \ref{aliensetDef} implies  that $x\nprec z''$ in $\P$, a contradiction. These arguments allow us to conclude that $\mathcal{R}_3$ is not peak-subposet of $\P$. In the same way, we can see that for all $n\geq 0$,  $\mathcal{R}_{4,n}$ is not a  peak-subposet of $\P$. 
\end{proof} 
A poset $\P$ is said to be \textit{locally of width $n$} or \textit{have local width $n$} if $n$ is the minimum integer such that for each $z\in\max\P$ it holds that $w(z_{\vartriangle})\leq n$. Clearly a poset of type $\mathbb{A}$ has local width  less than or equal to two. The following lemma describes sincere  posets of  type $\mathbb{A}$ and their socle-projective indecomposable modules.
\begin{lemma}\label{1}
Let $\P$ be a poset of type $\mathbb{A}$.  Then
\begin{itemize}
\item[(a)] $\md_{sp}k\P$ is of finite representation type.
\item[(b)] $\P$  is a sincere poset if and only if $\P$ is isomorphic to one of the 
following posets:\vspace{0.2cm}

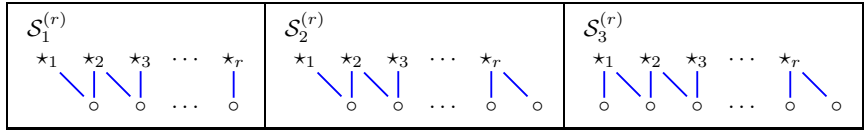
\begin{adjustbox}{max totalsize={0.9\textwidth}{0.9\textheight},center}
\begin{tabular}{|l|l|l|}
\hline 
\begin{tikzpicture}
 \node (1) at (0,0) {$\star_1$};
 
\node (A) at (0,0.5) {$\mathcal{S}^{(r)}_{1}$};
\node [right of=1, node distance=0.7cm](2) {$\star_2$};

\node [right of=2, node distance=0.7cm](3) {$\star_3$};
\node [right of=3, node distance=0.7cm] (4)  {$\cdots$};
\node [right of=4, node distance=0.7cm] (5) {$\star_{r}$};

\node [below  of=2, node distance=0.7cm] (2')  {$\circ$};
\node [below  of=3, node distance=0.7cm] (3')  {$\circ$};
\node [below  of=4, node distance=0.7cm] (4')  {$\dots$};
\node [below  of=5, node distance=0.7cm] (5')  {$\circ$};

   
    \draw [blue, thick, shorten <=-2pt, shorten >=-2pt] (2') -- (2);
    
\draw [blue,  thick, shorten <=-2pt, shorten >=-2pt] (3) -- (3');
    \draw [blue, thick, shorten <=-2pt, shorten >=-2pt] (5) -- (5');    
    
    \draw [blue, thick, shorten <=-2pt, shorten >=-2pt] (1) -- (2');
    
 \draw [blue, thick, shorten <=-2pt, shorten >=-2pt] (2) -- (3');
    
\end{tikzpicture} & 

\begin{tikzpicture}
 \node (1) at (0,0) {$\star_1$};
 
\node (A) at (0,0.5) {$\mathcal{S}^{(r)}_{2}$};
\node [right of=1, node distance=0.7cm](2) {$\star_2$};

\node [right of=2, node distance=0.7cm](3) {$\star_3$};
\node [right of=3, node distance=0.7cm] (4)  {$\cdots$};
\node [right of=4, node distance=0.7cm] (5) {$\star_{r}$};

\node [below  of=2, node distance=0.7cm] (2')  {$\circ$};
\node [below  of=3, node distance=0.7cm] (3')  {$\circ$};
\node [below  of=4, node distance=0.7cm] (4')  {$\dots$};
\node [below  of=5, node distance=0.7cm] (5')  {$\circ$};
\node [right of=5', node distance=0.7cm] (6')  {$\circ$};

   
    \draw [blue, thick, shorten <=-2pt, shorten >=-2pt] (2') -- (2);
    
\draw [blue,  thick, shorten <=-2pt, shorten >=-2pt] (3) -- (3');
    \draw [blue, thick, shorten <=-2pt, shorten >=-2pt] (5) -- (5');    
    
    \draw [blue, thick, shorten <=-2pt, shorten >=-2pt] (1) -- (2');
    
 \draw [blue, thick, shorten <=-2pt, shorten >=-2pt] (2) -- (3');
    
\draw [blue, thick, shorten <=-2pt, shorten >=-2pt] (6') -- (5);

\end{tikzpicture}

 & 
 
\begin{tikzpicture}
 \node (1) at (0,0) {$\star_1$};
 
\node (A) at (0,0.5) {$\mathcal{S}^{(r)}_{3}$};
\node [right of=1, node distance=0.7cm](2) {$\star_2$};

\node [right of=2, node distance=0.7cm](3) {$\star_3$};
\node [right of=3, node distance=0.7cm] (4)  {$\cdots$};
\node [right of=4, node distance=0.7cm] (5) {$\star_{r}$};

\node [below of=1, node distance=0.7cm] (1')  {$\circ$};
\node [below  of=2, node distance=0.7cm] (2')  {$\circ$};
\node [below  of=3, node distance=0.7cm] (3')  {$\circ$};
\node [below  of=4, node distance=0.7cm] (4')  {$\dots$};
\node [below  of=5, node distance=0.7cm] (5')  {$\circ$};
\node [right of=5', node distance=0.7cm] (6')  {$\circ$};

   \draw [blue,  thick, shorten <=-2pt, shorten >=-2pt] (1) -- (1');
    \draw [blue, thick, shorten <=-2pt, shorten >=-2pt] (2') -- (2);
    
\draw [blue,  thick, shorten <=-2pt, shorten >=-2pt] (3) -- (3');
    \draw [blue, thick, shorten <=-2pt, shorten >=-2pt] (5) -- (5');    
    
    \draw [blue, thick, shorten <=-2pt, shorten >=-2pt] (1) -- (2');
    
 \draw [blue, thick, shorten <=-2pt, shorten >=-2pt] (2) -- (3');
    
\draw [blue, thick, shorten <=-2pt, shorten >=-2pt] (6') -- (5);

\end{tikzpicture} 
 
  \\ 
\hline 
\end{tabular} 
\end{adjustbox}\vspace{0.2cm}
for some $r\geq 1$. Furthermore, the module $M=(M_x,{_y}h_x)_{x,y\in\P}$ in $\md_{sp}(k\P)$ such that $M_x=k$ for all $x\in\P$  and ${_y}h_x=\mbox{id}_k$ for each $x\preceq y$ in $\P$  is the unique sincere indecomposable object in $\md_{sp}(k\P).$
\end{itemize}
\end{lemma}

\begin{proof}
According to Theorem \ref{tipof} part (c), to prove the part (a) is enough to observe that no  poset listed in \cite[section 5]{simson3} is a peak-subposet of $\P$. Indeed, the posets of the series $\mathcal{P}_{2,n+1}$, $\mathcal{P}''_{2,n}$, $\mathcal{P}_{3,n}$, $n\geq 0$ and the poset $\mathcal{P}_{2,0}$ contain $\mathcal{R}_3$ as peak-subposet. Moreover, the posets of the series $\mathcal{P}'_{2,n+1}$, $\mathcal{P}''_{3,n}$, $n\geq 0$ contain  $\mathcal{R}_1$ as peak-subposet and the  posets of the series $\mathcal{P}'_{3,n}$, $n\geq 0$ contain $\mathcal{R}_2$ as peak-subposet. Note that by definition $\P$ does not contain as a peak-subposet  a poset of the series  $\mathcal{P}_{1,n}$, $n\geq 0$.   Moreover, we note that any poset of the form  $\lbrace \mathcal{P}_4,\dots,\mathcal{P}_{110}\rbrace$  contains as peak-subposet to  $\mathcal{R}_{i}$ for some $i=1,2,3$.\\  

In order to prove $(b)$, first we consider that $\P$ is one-peak poset. In this case, according to the list of sincere one-peak-posets (see \cite{kleiner75}) we have that $\P=\mathcal{S}^{(1)}_i$ for some $i=1,2,3$. Moreover,  we observe in the known lists of sincere $r$-peak posets of finite  type that  $\mathcal{F}^{(2)}_1=\mathcal{S}^{(2)}_1,\mathcal{F}^{(2)}_2=\mathcal{S}^{(2)}_2,\mathcal{F}^{(2)}_5=\mathcal{S}^{(2)}_3$ are  the sincere two-peak  posets of  type $\mathbb{A}$ (see \cite{justina}), $\mathcal{F}^{(3)}_{44}=\mathcal{S}^{(3)}_1$, $\mathcal{F}^{(3)}_{46}=\mathcal{S}^{(3)}_2,\mathcal{F}^{(3)}_{53}=\mathcal{S}^{(3)}_3$ are the sincere three-peak  posets  of type $\mathbb{A}$ (see \cite{justina1}) and $\mathcal{F}^{(r)}_{8}=\mathcal{S}^{(3)}_1$,$\mathcal{F}^{(r)}_{10}=\mathcal{S}^{(3)}_2,\mathcal{F}^{(r)}_{13}=\mathcal{S}^{(3)}_3$  are the sincere $r$-peak  posets of  type $\mathbb{A}$,  with $r\geq 4$ (see \cite{justina2}). Thus, the first  part of (b) is true. Now, we observe in the mentioned lists that for each $i=1,2,3.$ and for each $r\geq 1$ the sincere $r$-peak poset $\mathcal{S}^{(r)}_{i}$ has only one sincere prinjective indecomposable  $k\mathcal{S}^{(r)}_{i}$-module $M=(M_x,{_y}h_x)_{x,y\in\mathcal{S}^{(r)}_{i}}$ such that $M_x=k$ and ${_y}h_x=\mbox{id}_k$ for each $x\preceq y$. In this way, Lemma \ref{prin}  implies that the second part of (b) is true. 
\end{proof}

The following lemma will be used to prove the categorical equivalence proposed in Theorem \ref{omega}.
\begin{lemma} \label{2} Let $\P$ be a poset of type $\mathbb{A}$ associated to the quiver $Q^{F}$ as in Proposition \ref{0}. Then

\begin{itemize}
\item[(a)] Up to isomorphism, any indecomposable  $k\P$-module  $M=(M_x,{_y}h_x)_{x,y\in\P}$ in $\md_{sp}(k\P)$ is such that $M_x=k$ and ${_y}h_x=\mbox{id}_k$ for all $x\preceq y$ in $\supp M$.
\item[(b)]The support $\supp M$ of an indecomposable object in the category $\md_{sp}k\P$  is  connected as a subset of the quiver  $Q$.

\end{itemize}
\end{lemma}

\begin{proof}
Let $M=(M_x,{_y}h_x)_{x,y\in\P}$ be an indecomposable object in $\md_{sp}(k\P)$.  Then the image  $\bm\theta(M)=(\bm\theta(M)_x)_{x\in\P}$ of $M$ by the  adjustment functor $\bm\theta$ defined in Equation (\ref{adjustment}) is an indecomposable object in $\repr$.  Let $\mathcal{S}=\csupp(\bm\theta(M))$ be the coordinate support of $\bm\theta(M)$. Then the poset $\mathcal{S}$ is a peak-subposet of $\P$. Thus, Definition \ref{defposetypeA} implies that $\mathcal{S}$ is a  poset of type $\mathbb{A}$. Indeed, if $\mathcal{R}\in\{\mathcal{R}_1,\mathcal{R}_2,\mathcal{R}_3, \mathcal{R}_{4,n}\}$  is peak-subposet of $\mathcal{S}$ then $\mathcal{R}$ is peak-subposet of $\P$, a contradiction.   Also,  $\mathcal{S}$ is a sincere poset because the peak $\mathcal{S}$-space  $(\bm\theta(M)_x)_{x\in\mathcal{S}}$  is  sincere. Lemma \ref{1} part (b) implies that for some $r\geq 1$ and $i=1,2,3$, we have $\mathcal{S}=\mathcal{S}^{(r)}_i$. Moreover, by Proposition \ref{ind} we have that $\bm\theta(M)$ is isomorphic to $T_{\mathcal{S}}(L)$ where $L=(L_x)_{x\in\mathcal{S}}$ is a sincere peak $\mathcal{S}$-space in $\mathcal{S}$-spr  and $T_{\mathcal{S}}$ is the subposet induced functor defined in Equation (\ref{inducedfunctor}). Lemma \ref{prin} implies that $\bm\rho (L)=N=(N_x,{_y}g_x)_{x,y\in\mathcal{S}}$ is a sincere module in $\md_{sp}k\mathcal{S}$, and Lemma \ref{1} then proves that $N_x=k$, for all $x\in\mathcal{S}$ and ${_y}g_x=\mbox{id}_k$ for each $x\preceq y$ in $\mathcal{S}$. Hence,  Lemma \ref{prin} implies that $L=\bm \theta(N)$. Now, by definition of  $\bm \theta$ we have that  $L_z=k$ for all $z\in\max\mathcal{S}$. Let $\lbrace e_{z}\hspace{0.1cm} \vert\hspace{0.1cm}z\in\max\mathcal{S}\rbrace$ be  the standard basis  of the space $L^{\bullet}$,  then for all $x\in\mathcal{S}^{-}$ we have that  $L_x$ is the subspace of $L^{\bullet}$ generated by the vector $w_x=\sum_{z\succ x}e_z$. Moreover, let $\hat{L}=(\hat{L}_x)_{x\in\P}=T_{\mathcal{S}}(L)$ be the image of $L$ by the functor $T_{\mathcal{S}}$ then 
\[\hat{L}_{x}=
\begin{cases}
\langle w_x                                                                                                                                                                                                                                                                                                                                                                                                                                                                                                                                                                                        \rangle,  & \text{if }x\in(\max\mathcal{S})_{\vartriangle}\cap\mathcal{S}^{\triangledown}, \\
0, & \text{otherwise}.
\end{cases}\]

Thus, the image of $\hat{L}$ by the functor $\bm\rho$ is the $k\P$-module $M=M_{\hat{L}}=(\hat{L}_x,{_y}\pi_x)_{x,y \in\P}$ where ${_y}\pi_{x}: \hat{L}_x\to \hat{L}_y$ is such that $\lambda w_x \mapsto\lambda w_y$ if $x\preceq y$ in $(\max\mathcal{S})_{\vartriangle}\cap\mathcal{S}^{\triangledown}$ and ${_y}\pi_{x}=0$ if $x\preceq y$ and either $x$ or $y$ is not in $(\max\mathcal{S})_{\vartriangle}\cap\mathcal{S}^{\triangledown}$. It is easy to see that there is a  natural isomorphism between   $M_{\hat{L}}$ and the representation  described in (a).\\

To prove $(b)$ it is enough to see that the set $(\max\mathcal{S})_{\vartriangle}\cap\mathcal{S}^{\triangledown}$ is connected as a subset of the quiver $Q$ for any sincere peak-subposet $\mathcal{S}$ of $\P$. Note that  the poset $\mathcal{S}^{(r)}_1$ is a peak-subposet of  $\mathcal{S}=\mathcal{S}^{(r)}_i$ for all $i=1,2,3$. We suppose that $$\mathcal{S}^{(r)}_1=\lbrace z_1\succ x_2 \prec z_2, z_2\succ x_3 \prec z_3,\dots,z_{r-1}\succ x_r \prec z_{r} \rbrace$$ then $\lbrace z_1,\dots, z_r\rbrace\subseteq\max\P$ and since $\mathcal{R}_2\nsubseteq \P$ we have that $\lbrace x_2,\dots,x_r\rbrace\subseteq\min\P$. Thus, for each $z_i$, with $2\leq i\leq r-1$ the $z_i$-subquiver  $Q^{(z_i)}$ of $Q$ has the form $\xymatrix@=1.5em{x_i\ar[r]&\cdots\ar[r]&z_i&\cdots \ar[l]&x_{i+1}\ar[l]}$. Since each vertex in $Q^{(z_i)}_{0}$ belongs to the set $\lbrace z_{i}\rbrace_{\vartriangle}\cap \lbrace x_i,x_{i+1}\rbrace^{\triangledown}$,  then $Q^{(z_i)}_{0}\subset\Supp M$ for each $2\leq i\leq r-1$. Let $w$  (respectively $w'$) be the  left  (respectively right) extremal vertex of the quiver  associated to $\mathcal{S}$ and let $x$ (respectively $y$) be minimal element in $(x_2^{\triangledown}\cap(\P\setminus\mathcal{S}))\cup\lbrace w\rbrace $ (respectively $(x_r^{\triangledown}\cap(\P\setminus\mathcal{S}))\cup\lbrace w'\rbrace $ ) then it is easy to see that $\supp M=\left[x,y \right]_{Q}$, where $\left[x,y \right]_{Q}$ denote a interval of $Q$ which is a connected subset  of $Q$.\end{proof}


\section{Category of sp-diagonals}\label{Catdiagonals} In this section, we define a category  $\mathcal{C}_{(T,F)}$ of diagonals associated to a poset $\P$ of type $\mathbb{A}$ and we prove  in  Theorem \ref{omega} and Corollary \ref{corollary} that this category gives a geometric realization of the category of finitely generated socle-projective modules over the incidence $k-$algebra $k\P$.\\

Let $\P$ be a poset of type $\mathbb{A}$ associated to the quiver $Q^{F}$ as in Proposition \ref{0}. Thus,  $Q$ is a Dynkin quiver of type $\mathbb{A}_n$ and $F$ is an alien set for $Q$.  Let $T=\lbrace \tau_1,\dots,\tau_n\rbrace$ be the triangulation of a $(n+3)$-gon $\Pi_{n+3}$ such that $Q_T=Q$. A \textit {fan} in $T$ is a  maximal subset $\Sigma_v\subseteq T$ of at least two diagonals such that all  the diagonals in $\Sigma_v$  share the vertex $v$  of  $\Pi_{n+3}$. A diagonal $\tau\in \Sigma_v$ is said to be the \textit{peak-diagonal} of $\Sigma_v$ if it is maximal in $\Sigma_v$ in accordance with the order $\tau_x<\tau_y$ if and only if there is a path from the vertex $x$ to the vertex $y$ in the quiver $Q$. Geometrically, the peak-diagonal of a fan $\Sigma_v$  is the  diagonal  that can be obtained from each other diagonal in $\Sigma_v$ by a clockwise rotation around the vertex $v$ (see Figure \ref{fan}).
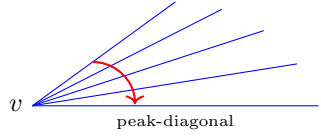
\begin{figure}[h]

\begin{center}
\begin{tikzpicture}
\tkzDefPoint(0,0){x}\tkzDefPoint(2,0){A}
\tkzDefPointsBy[rotation=center x angle 180/10](A,B,C,D,E,F,G,H,I){B,C,D,E,F,G,H,I,v}
\tkzLabelPoints[left](v)
\tkzDrawPolygon[thin,color=blue](v,B)
\tkzDrawPolygon[thin,color=blue](v,C)
\tkzDrawPolygon[thin,color=blue](v,D)
\tkzDrawPolygon[thin,color=blue](v,E)
\tkzDrawPolygon[thin,color=blue](v,F)
\tkzLabelSegment[below](v,B){{\tiny peak-diagonal}}
\draw [->, red, thick] (-1.1,1.2) arc (90:0:16pt);

\end{tikzpicture}
\end{center}
\caption{Fan of a triangulation} \label{fan}
\end{figure}
\begin{definition}\label{spdiagonal}
A diagonal $\gamma\notin T$ is  an \textit{sp-diagonal} if it satisfies the following conditions:

\begin{itemize}
\item[(a)] If $\gamma$  crosses  $\tau\in T$ then $\gamma$  crosses the peak-diagonal of a fan $\Sigma$ in $T$ such that $\tau\in\Sigma$. Henceforth, any diagonal $\gamma\notin T$ satisfying this condition will be called a \textit{$\star$-diagonal}.
\item[(b)]  For all alien arrows $\alpha\in F$ with $s(\alpha),t(\alpha)\in \text{supp }I(z)$, if $ \gamma$ crosses $\tau_{s(\alpha)}$ and $\tau_z$ then $\gamma$ also crosses $\tau_{t(\alpha)}$. Diagonals $\gamma\notin T$ satisfying this condition will be called  \textit{non-frozen} diagonals.
\end{itemize}
\end{definition}
\begin{example}\label{triangulation1}Let $Q$ be the quiver in Example \ref{zsubquiver} then $Q=Q_T$, where  $T$ is the following triangulation 

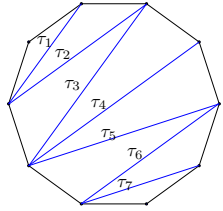
\begin{adjustbox}{scale=0.7,center}
\begin{tikzpicture}
\draw(0,0) node(w){\gon{2}{F}{D}{$\tau_1$}};
\draw(0,0) node(w){\gone{2}{F}{C}{$\tau_2$}};
\draw(0,0) node(w){\gone{2}{C}{G}{$\tau_3$}};
\draw(0,0) node(w){\gone{2}{G}{B}{$\tau_4$}};
\draw(0,0) node(w){\gone{2}{G}{A}{$\tau_5$}};
\draw(0,0) node(w){\gone{2}{A}{H}{$\tau_6$}};
\draw(0,0) node(w){\gone{2}{H}{J}{$\tau_7$}};
\end{tikzpicture}
\end{adjustbox}
 
In this case, the sets $\lbrace \tau_1,\bm\tau_2\rbrace$, $\lbrace \bm\tau_2,\tau_3\rbrace$, $\lbrace \tau_3,\tau_4,\bm\tau_5\rbrace$, $\lbrace \bm\tau_5,\tau_6\rbrace$ and $\lbrace \tau_6,\bm\tau_7\rbrace$ are fans of $T$. We have used bold font for the peak-diagonal of each fan. Note that, the  peak-diagonal corresponds to a sink vertex in the quiver $Q_T$. Moreover, let $Q^{F}$ be the quiver in the Example \ref{alien set}. Then, the diagonals 
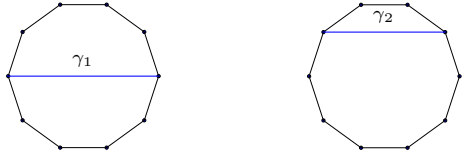
\begin{center}
\begin{tikzpicture}
\draw(4,0) node(w){\gon{1}{E}{B}{$_{\gamma_2}$}};
 \draw(0,0) node(w){\gon{1}{F}{A}{$_{\gamma_1}$}};
\end{tikzpicture} 
\end{center}

are such that $\supp\gamma_1=\lbrace 3,4\rbrace$ and $\supp\gamma_2=\lbrace 1,2,3\rbrace$. Thus, $\gamma_1$ is not a $\star$-diagonal  because it crosses $\tau_4$ but it does not cross the 
peak-diagonal $\tau_5$ in the unique fan $\lbrace \tau_3,\tau_4,\bm\tau_5\rbrace$ of $\tau_4$, whereas  $\gamma_2$ is a $\star$-diagonal because it crosses $\tau_2$ which is the peak-diagonal in the fans 
$\{\tau_1,\bm\tau_2\}$ and $\{\bm\tau_2,\tau_3\}$ for $\tau_1$, $\tau_2$ and $\tau_3$.\\

Given the alien arrows $\alpha:3\rightarrow 1$ and $\beta:6\rightarrow 4$ (see Example \ref{alien set}), a diagonal $\gamma$ is frozen by $\alpha$  if $\gamma$ crosses $\tau_3$ and $\tau_2$  but not $\tau_1$; whereas the diagonals frozen by $\beta$ cross $\tau_6$ and $\tau_5$  but not $\tau_4$ (see Figure \ref{frozendiagonals}).

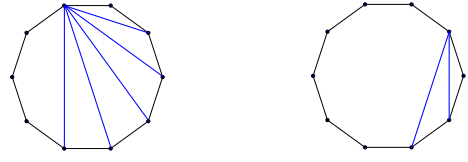
\begin{figure}[h]
  
\begin{center}
\begin{tikzpicture}
 \draw(0,0) node(w){\gon{1}{D}{B}{}};
\draw(0,0) node(w){\gone{1}{D}{A}{}};
\draw(0,0) node(w){\gone{1}{D}{H}{}};
\draw(0,0) node(w){\gone{1}{D}{I}{}};
\draw(0,0) node(w){\gone{1}{D}{J}{}};
\draw(4,0) node(w){\gon{1}{J}{B}{}};
\draw(4,0) node(w){\gone{1}{I}{B}{}};
\end{tikzpicture} 
\end{center}
\caption{Diagonals frozen by $\alpha$ (left) and  by $\beta$ (right).} \label{frozendiagonals}
\end{figure}
Note that, $\gamma_2$ is an sp-diagonal because it is non-frozen and  $\star$-diagonal.
\end{example}

The following lemma describes the relation between $\star$-diagonals and socle-projective modules in $\md kQ_T$.

\begin{lemma}\label{Ldiagonals} Let $\Theta:\mathcal{C}_T\to \md kQ_T$ be the equivalence of categories of Theorem~\ref{TRalf}, where  $Q_T$ is a Dynkin quiver of type $\mathbb{A}$. Then $\gamma$ is a $\star$-diagonal if and only if $\Theta(\gamma)$ is socle-projective.
\end{lemma}

\begin{proof} Since $Q_T$ is a Dynkin quiver of type $\mathbb{A}$, then $T=\{\tau_x\hspace{0.1cm}\vert\hspace{0.1cm} x\in (Q_T)_0\}$ is a triangulation without internal triangles. First, we suppose that $\gamma$ is a $\star$-diagonal. Let  $x$ be a vertex in $Q_T$ such that the indecomposable simple $kQ_T$-module  $S(x)$  at vertex $x$ is a submodule of $\Theta(\gamma)=M^{\gamma}$. We shall prove that $S(x)$ is a  projective $kQ_T$-module. Since $\Hom (S(x),M^{\gamma})\neq 0$, then $M^{\gamma}_x=k$, that is, $\tau_x$ crosses $\gamma$. By hypothesis, there exists a fan $\Sigma$ containing $\tau_x$ such that $\gamma$ crosses the peak-diagonal $\tau_z$ of  $\Sigma$. If $x\neq z$ then $\tau_x<\tau_z$, that is, there is a path $p$ in $Q_T$ from $x$ to $z$ whose vertices are in  $\supp\gamma$. Moreover, a nonzero morphism $f=(f_x)_{x\in (Q_T)_0}$ of representations  from $S(x)$ to $M^{\gamma}$ is such that $f_t=0$ for all $t\neq x$ because $S(x)$ is the simple representation at vertex $x$. Let $x\rightarrow y$ be the  arrow in $p$ starting in $x$, then the diagram  

\begin{center}
\begin{tikzcd}[row sep=large, column sep = large]
S(x)_x \arrow[r, "0"] \arrow[d,"f_x" ]
& S(x)_y \arrow[d, "0" ] \\
M^{\gamma}_x \arrow[r, "1" ]
& M^{\gamma}_y
\end{tikzcd}
\end{center}
commutes because $f$ is a morphism of representations of the quiver $Q_T$. Since $S(x)_y$ is zero and $M^{\gamma}_x=M^{\gamma}_y=k$, then $f_x=0$. Therefore, the morphism $f$ is zero, a contradiction. Thus, we conclude that $x=z$, that is, $\tau_x$ is a peak-diagonal. In other words, $x$ is a sink vertex in $Q_T$ and then $S(x)$ is projective. Since all simple submodules of $M^{\gamma}$ are projectives, we have that $\text{soc }M$ is projective.\\

In the other direction, we have that $\Theta(\gamma)$ is socle-projective. Let $\tau_x$ be a diagonal in $T$ crossing $\gamma$. If $\tau_x$ is a peak-diagonal then the definition  of $\star$-diagonal is trivially satisfied. If $\tau_x$ is not a peak-diagonal, we suppose that for all fans $\Sigma$ containing $\tau_x$, $\gamma$ does not cross the peak-diagonal in $\Sigma$.  We have that the number $s$ of fans containing $\tau_x$ is either one or two. In the case $s=1$, let $\tau_y$ be the maximal diagonal in the fan $\Sigma$  which crosses $\gamma$.  Then $\tau_x \leq\tau_y<\tau_z$, where $\tau_z$ is the peak-diagonal in $\Sigma$. In other words, there is a path $p$ from $x$ to $z$ in $Q_T$ passing by $y$,  such that the vertices $x,\dots,y$ in $p$ belong to $\supp\gamma$, whereas the others vertices in $p$ are not in $\supp\gamma$. In particular, $M^{\gamma}_x=M^{\gamma}_y=k$ and $M^{\gamma}_z=0$. Let $S(y)$ be the simple representation of $Q_T$ at vertex $y$. Because the diagram

\begin{center}
\begin{tikzcd}[row sep=large, column sep = large]
S(y)_x \arrow[r, "0"] \arrow[d,"0" ]
& S(y)_y \arrow[r, "0"]\arrow[d, "\lambda" ]& S(y)_z \arrow[d, "0" ] \\
M^{\gamma}_x \arrow[r, "1" ]
& M^{\gamma}_y \arrow[r, "0"] & M^{\gamma}_z
\end{tikzcd}
\end{center}
commutes, we conclude that there is a nonzero injective morphism from $S(y)$ to $M^{\gamma}$. Therefore, $S(y)$ is a non-projective module which is a submodule of $M^{\gamma}$, a contradiction  to the hypothesis.  In the case $s=2$, if $\tau_y$ (respectively $\tau_{y'}$) is the  maximal diagonal in $\Sigma$ (respectively $\Sigma'$)  crossing $\gamma$. By the above arguments, we  conclude that  $S(y)$ and $S(y')$ are non-projective summands of 
$\text{soc }M^{\gamma}$, a contradiction  to the hypothesis.  Therefore, $\gamma$ is a $\star$-diagonal.
\end{proof}

Let $\mathcal{C}_{(T,F)}$ be the full subcategory  of the category of diagonals $\mathcal{C}_T$ generated by all sp-diagonals in $\mathcal{C}_T$. We denote by $E(T,F)$ the set whose elements are the sp-diagonals in $\mathcal{C}_{(T,F)}$, the diagonals in $T$, and the boundary edges in $\Pi_{n+3}$.\\

Since the  irreducible morphisms in  $\mathcal{C}_{(T,F)}$  cannot be factorized through  sp-diagonals, we introduce the notion of a  \textit{pivoting sp-move} from  $\gamma\in E(T,F)$ to  $\gamma'\in E(T,F)$, that is, a composition of pivoting elementary moves of the form  
$$P:\gamma=\gamma_0\xrightarrow{P^{(1)}_v}\gamma_{1}\xrightarrow{P^{(2)}_v}\dots\xrightarrow{P^{(s)}_v}\gamma_{s}=\gamma'$$ with the same pivot $v$ such that $\gamma_1,\dots,\gamma_{s-1}$  are not sp-diagonals  in $\Pi_{n+3}$. Note that the irreducible morphisms in $\mathcal{C}_{(T,F)}$ are precisely the pivoting sp-moves between sp-diagonals.\\

Next, we analyze the relations in the category $\mathcal{C}_{(T,F)}$. These come from the mesh relations in $\mathcal{C}_T$.
 We suppose  that  $\gamma$ and $\gamma'$  are sp-diagonals and that the compositions  $\gamma\xrightarrow{P_1}\beta\xrightarrow{P_2}\gamma'$ and $\gamma\xrightarrow{P_3}\beta'\xrightarrow{P_4}\gamma'$ of two pivoting sp-moves are as in Figure \ref{mesh}. We  have that $$\gamma\xrightarrow{P_1}\beta\xrightarrow{P_2}\gamma'=\gamma\xrightarrow{P_3}\beta'\xrightarrow{P_4}\gamma',$$ 
taking into account the following convention: If one of the intermediate edges ($\beta$ or $\beta'$) is either a boundary edge or a diagonal in $T$, the corresponding term in the identity is replaced by zero.

\begin{figure}[h]
\begin{adjustbox}{scale=0.8,center}
\begin{tikzpicture}
\draw [] circle (2cm) ;
\tkzDefPoint(0,0){x}\tkzDefPoint(2,0){A}
\tkzDefPointsBy[rotation=center x angle 360/15](A,B,C,D,E,F,G,H,I,J,K,L,M,N){B,C,D,E,F,G,H,I,J,K,L,M,N,O}
\tkzDrawPolygon[thin,color=blue](G,O)
\tkzDrawPolygon[thin,color=blue](C,I)
\tkzDrawPolygon[thin,color=blue](C,G)
\tkzDrawPolygon[thin,color=blue](O,I)
\tkzLabelSegment[above](C,G){$\beta$}
\tkzLabelSegment[below](O,I){$\beta'$}

\draw [->, red, thick] (-.7,0.66) arc (-30:45:14pt);
\draw [->, red, thick] (-.7,-.5) arc (-30:45:16pt);

\draw (.5,0.3) node[below] {$\gamma$};
\draw (.6,1.1) node[below] {$\gamma'$};
\draw (.2,1.4) node[below, red, ultra thick] {$\bm P_2$};
\draw (-.4,1.2) node[below, red , ultra thick] {$\bm P_1$};
\draw (.2,-0.1) node[below, red, ultra thick] {$\bm P_3$};
\draw (-.4,0) node[below, red, ultra thick] {$\bm P_4$};

\tkzDrawPolygon[color=blue,loosely dotted](C,H)
\tkzDrawPolygon[color=blue, loosely dotted](G,B)
\tkzDrawPolygon[color=blue,loosely dotted](G,A)
\tkzDrawPolygon[color=blue,loosely dotted](I,B)
\tkzDrawPolygon[color=blue,loosely dotted](I,A)
\tkzDrawPolygon[color=blue,loosely dotted](O,H)
\draw [->, red, thick] (.5,1.4) arc (145:235:7pt);
\draw [->, red, thick] (.5,-0.1) arc (145:235:10pt);

\end{tikzpicture}
\end{adjustbox}
\caption{Mesh relations in $\mathcal{C}_{(T,F)}$.} \label{mesh}
\end{figure}
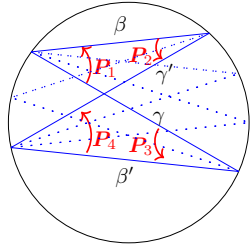

\subsection{The functor $\Omega$ } Let $\P$ be the  poset of type $\mathbb{A}$ associated to the quiver $Q^{F}$, where $Q$ is a quiver of Dynkin type $\mathbb{A}$ and $F$  an alien set for $Q$ and denote by $T$ a triangulation associated to $Q$.  Let us define a $k$-linear additive functor
$$\Omega:\mathcal{C}_{(T,F)}\to\md_{sp}(k\P)$$ from the category of sp-diagonals to the category of finitely generated socle-projective $k\P$-modules such that for  any sp-diagonal $\gamma$ we have $ \Omega(\gamma)=M^{\gamma}=(M_x^{\gamma},{_y}h^{\gamma}_{x})$
where $M^{\gamma}$ is  defined by the following identities:
    $$M_x^{\gamma}= \left\{ \begin{array}{ll}
k &   \mbox{if } x\in\mbox{supp}\hspace{0.1cm}\gamma, 
\\ 0 &  \mbox{otherwise.} 
\end{array}
\right. \hspace{2mm} \mbox{and if }x\preceq y\in\P\mbox{ then } {_y}h_{x}^{\gamma}= \left\{ \begin{array}{lc}
\mbox{id}_k &   \mbox{if }  x,y \in\supp\gamma,\\
 0 &  \mbox{otherwise.} 
\end{array}
\right.$$

Now, we define the functor $\Omega$ on morphisms. By additivity, it is sufficient to define the functor on morphisms between sp-diagonals. Our strategy is to define  the functor on pivoting  sp-moves and then check  that the mesh relations in $\mathcal{C}_{(T,F)}$ hold. For any pivoting sp-move $P:\gamma\to\gamma'$, we define the morphism $$\Omega(P)=~(\Omega(P)_x)_{x\in\P}:(M_x^{\gamma}, {_y}h^{\gamma}_x) \to  (M_x^{\gamma'}, {_y}h^{\gamma'}_x)$$ by the formula 
\[\Omega(P)_x=
\begin{cases}
\mathrm{id}_k,  &\text{if }M^{\gamma}_x=M^{\gamma'}_x=k, \\
0,            &\text{otherwise.}
\end{cases}\]

By definition, $\Omega$ maps compositions of pivoting sp-moves to  compositions  of the images of the pivoting sp-moves. Note that if $\P$ is the poset $\P_Q$ associated  to a Dynkin quiver $Q$ of type $\mathbb{A}$  (without alien arrows) then the functor $\Omega$ is the restriction of the functor $\Theta$ defined in  section \ref{Catdiagonals0} to the full subcategory of  $\mathcal{C}_{T}$ generated by the $\star$-diagonals in $\mathcal{C}_T$.\\
   
Now, we prove that the functor $\Omega$ is well-defined and that it is an equivalence of categories.

\begin{theorem}\label{omega}
$\Omega$ is an equivalence of categories.
\end{theorem}
\begin{proof}Recall here that $\P$ and $\P_Q$ are two different posets, that they have the same vertices and that $\P$ is obtained from $\P_Q$ by adding edges to the Hasse diagram corresponding to the alien arrows in $F$. In particular, $x\prec y$ in $\P_Q$ implies $x\prec y$ in $\P$.\\

In order to prove that $M^{\gamma}\in\md_{sp}(k\P)$ we have to proof  the conditions (a) and (b) in  Proposition \ref{sp module}. To prove   $(a)$ it is enough to consider the non trivial situation when $x\prec y\prec w $ in  $\P$ such that $x,w\in\supp\gamma$ and $y\notin\supp\gamma$. First we note that, by Lemma \ref{Lralf2}, if $x\prec y\prec w$ in $\P_Q$ then $y\in\supp\gamma$ which is contradictory.  Therefore, we have that  $x\npreceq y$ or $y\npreceq w$ in $\P_Q$. We consider the following cases (recall that we always suppose $x\prec y\prec w $ in  $\P$ such that $x,w\in\supp\gamma$ and $y\notin\supp\gamma$).\\

\textbf{Case 1} $y\prec w$ in  $\P_Q$ and $x\npreceq y$ in $\P_Q$. In this case,  there exists an alien arrow $\alpha:x'\rightarrow y'$ on vertices of a $z$-subquiver $Q^{(z)}$ of $Q$ such that $x\preceq x'\prec z$ and  $y'\preceq y\prec w\prec z$ in $\P_Q$. Since $w\in\supp \gamma$ and $\tau_z$ is the maximal diagonal in the unique fan of $\tau_w$ then  $\gamma$ crosses $\tau_{z}$. Thus, since $x\in\supp\gamma$ Lemma \ref{Lralf2} implies that $\gamma$ crosses $\tau_{x'}$,   and since  $\gamma$ is non-frozen, then $\gamma$ crosses $\tau_{y'}$. Again using  Lemma \ref{Lralf2}, we obtain that $\gamma$ crosses $\tau_{y}$, that is, $y\in\supp\gamma$ which is contradictory.\\

\textbf{Case 2} $x\prec y$ in  $\P_Q$ and  $y\npreceq w$ in $\P_Q$. In this case,  there exists an alien arrow $\alpha:y'\to w'$ on vertices of a $z$-subquiver $Q^{(z)}$ such that $x\prec y\preceq y'\prec z$ and  $w'\preceq w\prec z$ in $\P_Q$. Since $w\in\supp \gamma$ and $\tau_z$ is the maximal diagonal in the unique fan of $\tau_w$,  then  $\gamma$ crosses $\tau_{z}$. Thus, since $x\in\supp\gamma$ Lemma \ref{Lralf2} implies that $\gamma$ crosses $\tau_y$, in other words, $y\in\supp\gamma$ which cannot be.\\

\textbf{Case 3} $x\npreceq y$ and $y\npreceq w$ in $\P_Q$. In this case,  there exist two alien arrows $\alpha:x'\to y'$ and $\alpha:y''\to w'$  on vertices of a $z$-subquiver $Q^{(z)}$ of $Q$
such that $x\preceq x'\prec w'\preceq w\prec z$ and  $y'\preceq y\preceq y''\prec z$ in $\P_Q$. Since $w\in\supp \gamma$ and $\tau_z$ is the maximal diagonal in the unique fan of $\tau_w$, then  $\gamma$ crosses $\tau_{z}$. Thus,  Lemma \ref{Lralf2} implies that $\gamma$ crosses $\tau_{x'}$, and since  $\gamma$ is non frozen we conclude that $\gamma$ crosses $\tau_{y'}$. Again using  Lemma \ref{Lralf2} we obtain that $\gamma$ crosses $\tau_{y}$, that is, $y\in\supp\gamma$ which is contradictory.\\

We have shown that if  $x\prec y\prec w $ in  $\P=\P_{Q^F}$ such that $x,w\in\supp\gamma$ then  $y\in\supp\gamma$. Thus, ${_w}h^{\gamma}_y={_y}h^{\gamma}_x={_w}h^{\gamma}_x=\mbox{id}_k$ and condition (a) holds. To prove condition $(b)$, let $x$ be an element of $\P^{-}=\P\setminus\max\P$. If $x\notin\supp\gamma$ then clearly $\ker {_z}h^{\gamma}_x=0$ for all $z\in\max\P$ such that $x\prec z$. If $x\in\supp\gamma$ then ${_{z'}}h^{\gamma}_x=\mathrm{id}_k$ for some $z'\in\max\P$, where $\tau_{z'}$ is the peak-diagonal in some fan  containing  $\tau_x$. Thus, $\stackbin[z\in\max\P]{}{\bigcap}\ker {_z}h_x=0$ for all $x\in\P^{-}$ such that $x\prec z$.  This shows that $\Omega(\gamma)=M^{\gamma}$ is indeed an object in $\md_{sp}(k\P)$.\\

Let us now check that $\Omega(P)$ is well defined for every pivoting sp-move $P:\gamma\to\gamma'$. Indeed, it is enough to show that for any relation $x\prec y$ such that $y$ covers $x$ in $\P$ the diagram 
\[\begin{tikzcd}[row sep=large, column sep = large]
M^{\gamma}_x \arrow[r, "{_y}h^{\gamma}_x"] \arrow[d,"\Omega(P)_x" ]
& M^{\gamma}_y \arrow[d, "\Omega(P)_y" ] \\
 M^{\gamma'}_x \arrow[r, "{_y}h^{\gamma'}_x" ]& M^{\gamma'}_y \end{tikzcd}\] 
 commutes. Note that the result holds if $M^{\gamma}_x=0$ or $M^{\gamma'}_y=0$ and also if both $M^{\gamma}_y$ and $M^{\gamma'}_x$ are null spaces. Suppose now that $M^{\gamma}_x=M^{\gamma'}_y=k$. If $M^{\gamma}_y=M^{\gamma'}_x=k$, then all four maps are $\mbox{id}_k$ and the diagram commutes. 
 The only remaining case is if exactly one of  $M^{\gamma}_y$, $M^{\gamma'}_x$ is nonzero. We will show that this cannot happen. 
 Suppose that $M^{\gamma'}_x=0$ and  $M^{\gamma}_y=k$, that is, $x,y\in\supp\gamma$, $y\in\supp\gamma'$ and $x\notin\supp\gamma'$. Since $y$ covers $x$ in $\P$, there exists an arrow $\alpha:x\to y$ in $Q^F$. If $x\prec y$ in $\P_Q$ then $\alpha$ is an arrow in $Q$,
 that is, $\tau_x$ and $\tau_y$ share a vertex of the polygon and are connected by a pivoting elementary move.  
 Since $P:\gamma\to\gamma'$ is a pivoting sp-move we get that $\tau_x$ crosses $\gamma$, that $\tau_x$ and $\gamma'$ have a common point on the boundary of the polygon and that $\tau_y$ crosses $\gamma$ and $\gamma'$. This implies that $\tau_y$ is clockwise from $\tau_x$ and that contradicts the orientation $x\rightarrow y$ in the quiver $Q$ (see Figure \ref{counterclockwise}). 
 Next, we suppose that $x\nprec y$ in $\P_Q$,
  then $\alpha:x\to y$ is an alien arrow in $F$ 
  with $x$ and $y$ in $\text{Supp }I(z)$ for some sink vertex $z$ in $Q$. Now, by Definition \ref{aliensetDef} part (b),  $y$ is not a source vertex in $Q$ unless $y$ is an extremal vertex in $Q$. Thus, there is at most one arrow in $Q$ with starting point $y$, and therefore there is exactly on fan $\Sigma$ containing $\tau_y$ and $\tau_z$ is its peak-diagonal. By Definition~\ref{spdiagonal}, both $\gamma$ and $\gamma'$ cross $\tau_z$, because they are $\star$-diagonals crossing $\tau_y$. \\
  
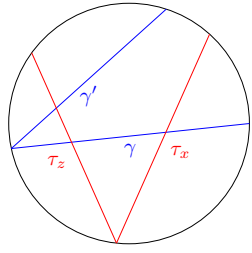
\begin{adjustbox}{scale=0.8,center}
\begin{tikzpicture}
\draw [] circle (2cm) ;
\tkzDefPoint(0,0){x}\tkzDefPoint(2,0){A}
\tkzDefPointsBy[rotation=center x angle 360/15](A,B,C,D,E,F,G,H,I,J,K,L,M,N){B,C,D,E,F,G,H,I,J,K,L,M,N,O}
\tkzDrawPolygon[thin,color=blue](D,I)
\tkzDrawPolygon[thin,color=blue](A,I)
\tkzDrawPolygon[thin,color=red](L,C)
\tkzDrawPolygon[thin,color=red](L,G)

\tkzLabelSegment[below, blue](A,I){$\gamma$}
\tkzLabelSegment[below, blue](D,I){$\gamma'$}

\tkzLabelSegment[below right, red](L,C){$\tau_x$}
\tkzLabelSegment[below left, red](L,G){$\tau_z$}
\end{tikzpicture}
\end{adjustbox}\\

On the other hand, there is a pivoting path from $\tau_z$ to $\tau_x$ in $\Pi_{n+3}$, since $x$ belongs to $\text{Supp }I(z)$. But this is impossible, because if $\tau\to\tau_x$ is  a pivot, then $\tau$ does not cross $\gamma'$. The other case where $M^{\gamma'}_x=k$ and  $M^{\gamma}_y=0$ is proved in a similar way.\\

To show that the functor $\Omega$ is well defined, it only remains to check  the  mesh relations. Indeed, let $\gamma\xrightarrow{P_1} \beta$, $\beta\xrightarrow{P_2} \gamma'$, $\gamma\xrightarrow{P_3}\beta'$, $\beta'\xrightarrow{P_4}\gamma'$ be pivoting sp-moves as in Figure \ref{mesh} with  $\gamma$,$\gamma'$ sp-diagonals and $\beta\neq\beta'$ sp-diagonals, diagonals in $T$ or boundary edges. Note that, we can exclude the case where $\beta$ and $\beta'$ are both diagonals in the triangulation $T$ or both boundary edges because in this case either $\gamma$ or $\gamma'$ is a diagonal in $T$. Without loss of generality, we may assume from now on that $\beta$ is an sp-diagonal.  Suppose first that $\beta'$ is an sp-diagonal; then one has to check the commutativity of the following diagram
\begin{center}
\begin{tikzcd}[row sep=large, column sep = large]
M^{\gamma}_x \arrow[r, "\Omega(P_1)_x"] \arrow[d,"\Omega(P_3)_x" ]
& M^{\beta}_x \arrow[d, "\Omega(P_2)_x" ] \\
M^{\beta'}_x \arrow[r, "\Omega(P_4)_x" ]
& M^{\gamma'}_x
\end{tikzcd}
\end{center}
for all $x\in\P$. Again, the only non trivial case happens when $M_x^{\gamma}=M_x^{\gamma'}=k$. In this case we also have $M_x^{\beta}=M_x^{\beta'}=k$ because any diagonal crossing both $\gamma$ and $\gamma'$ must also  crosses $\beta$ and $\beta'$. Thus all maps are $\mbox{id}_k$ and the diagram commutes. Suppose now that $\beta'$ is  a boundary edge or diagonal in $T$.
 Then  we have to show that the composition $M_x^{\gamma}\xlongrightarrow{\Omega(P_1)} M_x^{\beta}\xlongrightarrow{\Omega(P_2)} M_x^{\gamma'}$ is zero for all $x\in\P$.  Clearly if $\beta'$ is a boundary edge or diagonal in $T$ then no diagonal $\tau\in T$ can cross both $\gamma$ and $\gamma'$ then $\Hom(\Omega(\gamma), \Omega(\gamma'))=0.$ \\

In order to prove that $\Omega$ is dense we fix an indecomposable $M\in\md_{sp}(k\P)$. Then by Lemma \ref{2} part (b), Lemma \ref{Lralf2} and Theorem \ref{TRalf} part (a) there exists a diagonal $\gamma\notin T$ such that $\supp \gamma= \supp M$. We show that $\gamma$ is an sp-diagonal. Indeed, since the socle of $M$ is projective, Lemma \ref{Ldiagonals} implies that $\gamma$ is a $\star$-diagonal.  Moreover, given an alien arrow $\alpha:x\rightarrow y$ in $F$, with $x$ and $y$ in $\text{Supp }I(z)$ for some sink vertex $z$ in $Q_0$  such that $x,z\in\supp M$ then ${_z}h_x=\mbox{id}_k$. By Proposition \ref{sp module} part $(a)$, we have that ${_z}h_x={_z}h_y \cdot {_y}h_x$, thus $y\in\supp M$.  Therefore $\gamma$ crosses $\tau_y$ and thus  $\gamma$ is a non-frozen diagonal. We conclude that $\gamma$ is an sp-diagonal and that $\Omega(\gamma)=M$.\\

To show that  $\Omega$ is faithful, it is enough to prove that the image of a nonzero morphism between sp-diagonals is a nonzero morphism in  $\md_{sp}(k\P)$. Indeed, let $P\in\Hom_{\mathcal{C}_{(T,F)}}(\gamma,\gamma')$ be a nonzero morphism in $\mathcal{C}_{(T,F)}$. Then $P$ also is a nonzero morphism in $\mathcal{C}_T$. Lemma \ref{Lralf1} implies that there exists a diagonal $\tau_x\in T$ crossing $\gamma$ and $\gamma'$ as in Figure \ref{relativeposition}. In particular, $M_x^{\gamma}=M_x^{\gamma'}=k$, and therefore $\Omega(P)_x=\mbox{id}_k\neq 0$.\\

Finally, we show that functor $\Omega$ is full. To do  so, let $\Omega(\gamma)\xlongrightarrow{g} \Omega(\gamma') $ be a nonzero morphism in $\md_{sp}(k\P)$. Then $g=(g_x)_{x\in Q_0}$, where $g_x$ is a linear map  from $\Omega(\gamma)_x$ to $\Omega(\gamma')_x$.   The map $\hat{g}=(\hat{g}_x)_{x\in Q_0}$ from $\Theta(\gamma)$ to $\Theta(\gamma')$ such that $\hat{g}_x=g_x$ is a morphism of representations in $\md kQ$. Indeed,  for each arrow $\alpha:x\rightarrow y$ in $Q_1$, we have $x\prec y$ in $\P$. Since $g$ is morphism in $\md_{sp}k\P$, then the diagram 
\begin{center}
\begin{tikzcd}[row sep=large, column sep = large]
\Omega(\gamma)_x \arrow[r,"{_y}h^{\gamma}_x"] \arrow[d,"g_x" ]& \Omega(\gamma)_y \arrow[d, "g_y" ] \\
\Omega(\gamma')_x \arrow[r, "{_y}h^{\gamma'}_x" ]
& \Omega(\gamma')_y
\end{tikzcd}
\end{center}
commutes. Note that the elements in $\P$ are the vertices in $Q_0$.  Moreover, if $\gamma $ is an sp-diagonal then
  the representations $\Theta(\gamma)=~(\Theta(\gamma)_x, f^{\gamma}_{\alpha})$ in $\md kQ$ and $\Omega(\gamma)=(\Omega(\gamma)_x,{_y}h^{\gamma}_x )$ in $\md_{sp}(k\P)$ have the same $k$-vector spaces $\Omega(\gamma)_x= \Theta(\gamma)_x$ for all $x\in\P$ and  the same maps ${_y}h^{\gamma}_x=f^{\gamma}_{\alpha}$  for each $\alpha:x\rightarrow y$ in $Q_1$ (the map $f^{\gamma}_{\alpha}$ is not defined when $\alpha$ is an alien arrow for $Q$). 
  Thus  we have a commutative diagram  
\begin{center}
\begin{tikzcd}[row sep=large, column sep = large]
\Theta(\gamma)_x \arrow[r,"f^{\gamma}_\alpha"] \arrow[d,"\hat{g}_x" ]& \Theta(\gamma)_y \arrow[d, "\hat{g}_y" ] \\
\Theta(\gamma')_x \arrow[r, "f^{\gamma'}_\alpha" ]
& \Theta(\gamma')_y
\end{tikzcd}
\end{center}
and hence the map $\hat{g}$ is a morphism in $\md kQ$. Under the equivalence of categories $\Theta:\mathcal{C}_T\to\md kQ_T$ of Theorem \ref{TRalf}, the morphism $\hat{g}$ corresponds to a morphism $P\in\Hom_{\mathcal{C}_T}(\gamma,\gamma')$, with $\Theta(P)=\hat{g}$. Since  $\gamma$ and $\gamma'$ are sp-diagonals in $\mathcal{C}_T$,  $P$  also is a morphism in  the full subcategory $\mathcal{C}_{(T,F)}$ of  $\mathcal{C}_T$.  The definition of the functors  $\Theta$ and $\Omega$ on morphisms implies that $\Omega(P)=g$.\end{proof}

The following corollary is an direct consequence of the arguments used in Theorem~\ref{omega} and section \ref{Catdiagonals}.
\begin{corollary}\label{corollary}
Let $\P$ be a poset of type $\mathbb{A}$ associated to the  quiver $Q^{F}$  as in Proposition \ref{0} and let $\mathcal{C}_{(T,F)}$ be the corresponding category of sp-diagonals. Then
\begin{itemize}
\item[(a)] The irreducible morphisms of $\mathcal{C}_{(T,F)}$ are direct sums of the generating morphisms given by pivoting sp-moves.
\item[(b)] Let $\gamma\xlongrightarrow{P_1}\beta\xlongrightarrow{P_2}\gamma'$ be a composition of two pivoting sp-moves as in  Figure~\ref{mesh}, where $\gamma$, $\gamma'$, and $\beta$ are sp-diagonals. Then 
\begin{itemize}
\item[(i)] The sequence $0\longrightarrow \gamma\longrightarrow \beta\oplus\beta'\longrightarrow \gamma'\longrightarrow 0$ is an AR-sequence if
 $\beta'$ is a sp-diagonal.
\item[(ii)] The sequence $0\longrightarrow \gamma\longrightarrow \beta\longrightarrow \gamma'\longrightarrow 0$ is an AR-sequence if $\beta'$ is either a boundary edge or a diagonal in $T$.
\item[(iii)] If $\beta'\notin E(T,F)$  then $\gamma'$ is an indecomposable projective in $\mathcal{C}_{(T,F)}$ and $\gamma$ is  an indecomposable injective in $\mathcal{C}_{(T,F)}$.
\end{itemize}
\end{itemize} 
\end{corollary}

\begin{example} Let $Q$ and $F$ as in Example \ref{alienset1}. Then the triangulation $T$ associated to $Q$ has the form \\

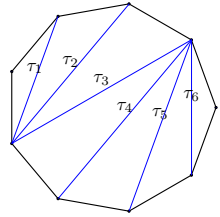
\begin{adjustbox}{scale=0.7,center}
\begin{tikzpicture}
\draw(0,0) node(w){\gonu{2}{F}{D}{$\tau_1$}};
\draw(0,0) node(w){\gonuo{2}{F}{C}{$\tau_2$}};
\draw(0,0) node(w){\gonuo{2}{F}{B}{$\tau_3$}};
\draw(0,0) node(w){\gonuo{2}{G}{B}{$\tau_4$}};
\draw(0,0) node(w){\gonuo{2}{B}{H}{$\tau_5$}};
\draw(0,0) node(w){\gonuo{2}{B}{I}{$\tau_6$}};

\end{tikzpicture}
\end{adjustbox}
The AR-quiver $\Gamma(\mathcal{C}_T)$ of the category $\mathcal{C}_T$ has the shape  \\

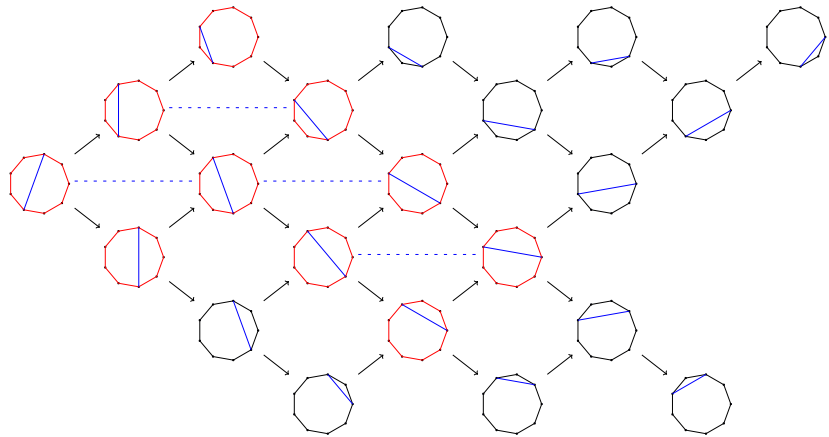
\begin{adjustbox}{scale=0.4,center}
\begin{tikzcd}[ampersand replacement=\&, row sep= tiny, column sep = normal]
\& \& \gonured{1}{E}{G}{}\arrow[rd]\&\& \gonu{1}{H}{F}{}\arrow[rd] \&\& \gonu{1}{I}{G}{}\arrow[rd]\&\& \gonu{1}{H}{A}{}\\ \& \gonured{1}{D}{G}{}\arrow[rd]\arrow[ru]\arrow[rr,loosely dashed, dash,blue] \&\& \gonured{1}{E}{H}{}\arrow[rd]\arrow[ru] \&\& \gonu{1}{I}{F}{}\arrow[rd]\arrow[ru]\&\& \gonu{1}{A}{G}{}\arrow[ru]\\ \gonured{1}{C}{G}{}\arrow[rd]\arrow[ru]\arrow[rr,loosely dashed, dash,blue] \&\& \gonured{1}{D}{H}{}\arrow[rd]\arrow[ru] \arrow[rr,loosely dashed, dash,blue] \&\& \gonured{1}{E}{I}{}\arrow[rd]\arrow[ru] \&\& \gonu{1}{F}{A}{}\arrow[ru] \\  \& \gonured{1}{C}{H}{}\arrow[rd]\arrow[ru] \&\& \gonured{1}{D}{I}{}\arrow[rd]\arrow[ru]\arrow[rr,loosely dashed, dash,blue] \&\& \gonured{1}{E}{A}{}\arrow[rd]\arrow[ru]\\ \&\& \gonu{1}{C}{I}{}\arrow[ru]\arrow[rd]\& \&  \gonured{1}{D}{A}{}\arrow[ru]\arrow[rd]\&\&  \gonu{1}{E}{B}{}\arrow[rd] \\ 
\& \& \& \gonu{1}{C}{A}{}\arrow[ru]\&\& \gonu{1}{D}{B}{}\arrow[ru]\&\& \gonu{1}{E}{C}{}
 \end{tikzcd}
\end{adjustbox}
\\

Here, we have drawn   the polygons with sp-diagonals using red color, that is, the diagonals   $\gamma$  such that $\gamma$ crosses  $\tau_3$ and if $\gamma$ crosses $\tau_5$ then $\gamma$ crosses $\tau_2$. Hence, the AR-quiver $\Gamma(\mathcal{C}_{(T,F)})$ of the category $\mathcal{C}_{(T,F)}$ is the red part of $\Gamma(\mathcal{C}_T)$, where  dotted lines have been drawn  to describe the action of  the AR-translation in $\Gamma(\mathcal{C}_{(T,F)})$. 
 \end{example}

\begin{example} Let $\P=\P_{Q^F}$  be the three-peak poset  of  type $\mathbb{A}$ defined in Example \ref{posetstypeA} which is the poset associated to the quiver  $Q^F$  in Example \ref{alien set}. Recall that,  $\P$ can be viewed as a Dynkin quiver of type $\mathbb{E}_7$.  Thus, the AR-quiver  $\Gamma(\md(k\P))$ of  the module category $\md (k\P)$ can be built using the knitting algorithm (see \cite{schiffler}) and it has the form\\

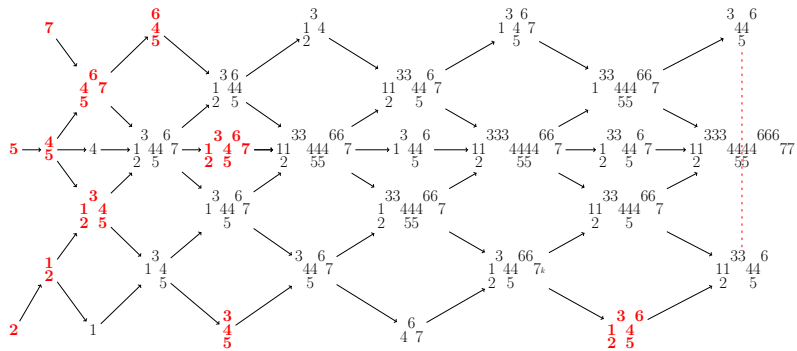
\begin{adjustbox}{scale=0.3,center, rotate = 0}    
\begin{tikzcd}[ampersand replacement=\&, row sep=normal, column sep =normal]
\&\smrd{\colorb{7}}\&\&\smrd{\colorb{6}\\\colorb{4}\\\colorb{5}} \&\& \smrd{& 3\\ 1&& 4\\ 2} \&\&  \smrd{&3&&6\\1&&4&&7\\&&5}\&\& \sm{3&&6\\&44\\&5}\arrow[dddd ,loosely dashed, dash, red]\\ %
\&\&\smrud{& \colorb{6}\\\colorb{4}&& \colorb{7}\\  \colorb{5} }\&\&\smrud{&  3 &  6\\  1 && 44\\  2&&  5} \&\& \smrud{&33&&6\\11&&44&&7\\2&&5}\&\& \smrud{&33&&66\\1&&444&&7\\&&55} 
\\
\smr{\colorb{5}}\&\smrud{\colorb{4}\\ \colorb{5}}\arrow[r]\& \smr{4}\&\smrud{& 3&& 6\\ 1&&44&&  7\\  2&& 5}\arrow[r]\& \smr{&\colorb{3}&&\colorb{6}\\\colorb{1}&&\colorb{4}&&\colorb{7}\\\colorb{2}&&\colorb{5}}\arrow[r]\&\smrud{& 33&& 66\\11 && 444&&7\\2&& 55}\arrow[r]\& \sm{&3&&6\\1&&44\\&&5}\arrow[r] \& \smrud{&333&&66\\11&&4444&&7\\2&&55}\arrow[r]\& \sm{&33&&6\\1&&44&&7\\2&&5} \arrow[r] \& \sm{&333&&666\\11&&4444&&77\\2&&55} \\%
\&\&\smrud{& \colorb{3} \\\colorb{1}&& \colorb{4}\\\colorb{2}&& \colorb{5}} \&\& \smrud{&3&& 6\\1 && 44&& 7\\  && 5}\&\& \smrud{&33&&66\\1&&444&&7\\2&&55}\&\& \smrud{&33&&66\\11&&444&&7\\2&&5}\\%
\&\smrud{\colorb{1}\\\colorb{2}}\&\&\smrud{&3\\1&&4\\&&5}\&\&\smrud{3&&6\\&44&&7\\&5}\&\& \smrud{&3&&66\\1&&44&&7\\2&&5}k \&\& \sm{&33&&6\\11&&44\\2&&5}
\\
\smru{\colorb{2}}\&\& \smru{1}\&\& \smru{\colorb{3}\\\colorb{4}\\\colorb{5}}\&\& \smru{&6\\4&&7}\&\& \smru{&\colorb{3}&&\colorb{6}\\\colorb{1}&&\colorb{4}\\\colorb{2}&&\colorb{5}}
\end{tikzcd}
\end{adjustbox}

\vspace{.5cm}

\begin{adjustbox}{scale=0.32,center, rotate = 0}    
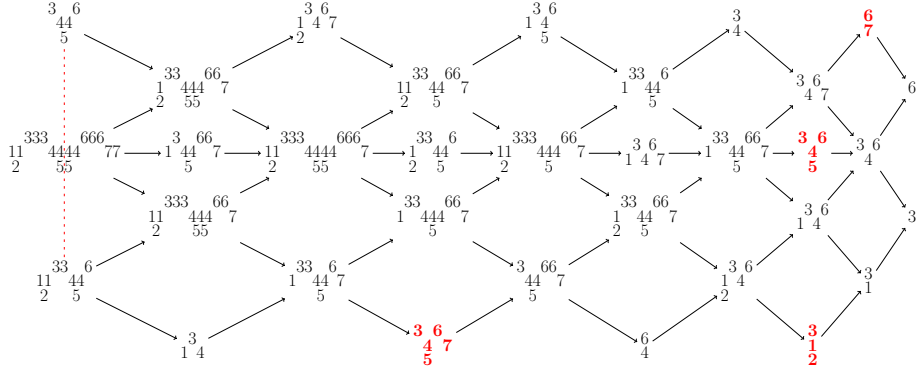
\begin{tikzcd}[ampersand replacement=\&, row sep=normal, column sep =normal]
\smrd{3&&6\\&44\\&5}\arrow[dddd ,loosely dashed, dash, red]\&\& \smrd{&3&&6\\1&&4&&7\\2}\&\& \smrd{&3&&6\\1&&4\\&&5}\&\& \smrd{3\\4}\&\& \smrd{\colorb{6}\\\colorb{7}}  \\ %
\&  \smrud{&33&&66\\1&&444&&7\\2&&55}\&\&\smrud{&33&&66\\11&&44&&7\\2&&5}\&\& \smrud{&33&&6\\1&&44\\&&5}\&\& \smrud{3&&6\\&4&&7}\&\& \sm{6}
\\
 \smrud{&333&&666\\11&&4444&&77\\2&&55}\arrow[r] \& \sm{&3&&66\\1&&44&&7\\&&5}\arrow[r]\& \smrud{&333&&666\\11&&4444&&7\\2&&55}\arrow[r]\& \sm{&33&&6\\1&&44\\2&&5}\arrow[r] \&\smrud{&333&&66\\11&&444&&7\\2&&5}\arrow[r] \& \sm{&3&&6\\1&&4&&7}\arrow[r] \& \smrud{&33&&66\\1&&44&&7\\&&5}\arrow[r] \& \sm{\colorb{3}&&\colorb{6}\\&\colorb{4}\\& \colorb{5}}\arrow[r]\& \smrud{3&&6\\&4}    \\%
\& \smrud{&333&&66\\11&&444&&7\\2&&55}\&\&  \smrud{&33&&66\\1&&444&&7\\&&5}\&\& \smrud{&33&&66\\1&&44&&7\\2&&5}\&\& \smrud{&3&&6\\1&&4} \&\& \sm{3} \\%
\smrud{&33&&6\\11&&44\\2&&5}\&\& \smrud{&33&&6\\1&&44&&7\\&&5}\&\& \smrud{3&&66\\&44&&7\\&5}\&\& \smrud{&3&&6\\1&&4\\2}\&\& \smru{3\\1}
\\
\& \smru{&3\\1&&4} \&\&  \smru{\colorb{3}&&\colorb{6}\\&\colorb{4}&&\colorb{7}\\&\colorb{5}}\&\& \smru{6\\4}\&\& \smru{\colorb{3}\\\colorb{1}\\\colorb{2}} 
\end{tikzcd}
\end{adjustbox}

In the diagram, we have drawn  the dimensions of indecomposable socle-projective modules with red color. Hence, the AR-quiver $\Gamma(\md_{sp}(k\P))$ of the category of finitely generated socle-projective modules $\md_{sp}(k\P)$ has the form\\

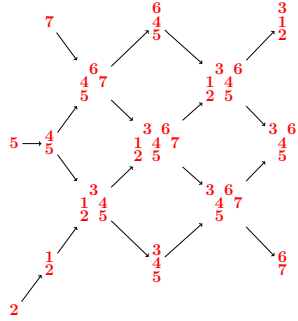
\begin{adjustbox}{scale=0.3,center}
\begin{tikzcd}[ampersand replacement=\&, row sep= normal, column sep = normal]
\&\smrd{ \colorb{7}}\&\&\smrd{ \colorb{6}\\ \colorb{4}\\ \colorb{5}} \&\& \sm{ \colorb{3}\\ \colorb{1}\\ \colorb{2}}\\
\&\&\smrud{&  \colorb{6}\\ \colorb{4}&&  \colorb{7}\\  \colorb{5} }\&\&\smrud{& \colorb{3}&&  \colorb{6}\\ \colorb{1}&&  \colorb{4}\\ \colorb{2}&&  \colorb{5}} \&\\
\smr{ \colorb{5}}\&\smrud{ \colorb{4}\\ \colorb{5}} \& \&\smrud{&  \colorb{3}&&  \colorb{6}\\ \colorb{1}&&  \colorb{4}&&  \colorb{7}\\ \colorb{2}&&  \colorb{5}}\&\&\sm{& \colorb{3}&&  \colorb{6}\\ &&  \colorb{4}\\ &&  \colorb{5}}\\
\&\&\smrud{&  \colorb{3} \\ \colorb{1}&&  \colorb{4}\\ \colorb{2}&&  \colorb{5}} \&\& \smrud{ \colorb{3}&&  \colorb{6}\\ &  \colorb{4}&&  \colorb{7}\\  &  \colorb{5}}\&\\
\&\smru{ \colorb{1}\\ \colorb{2}}\&\&\smru{ \colorb{3}\\ \colorb{4}\\ \colorb{5}}\&\&\sm{ \colorb{6}\\ \colorb{7}}\\
\smru{ \colorb{2}}
\end{tikzcd}
\end{adjustbox}

On the other hand, since the triangulation $T$ associated to the quiver $Q$ was described  in Example \ref{triangulation1},   
the AR-quiver $\Gamma(\mathcal{C}_T)$ of the category of diagonals $\mathcal{C}_T$  has the form \\

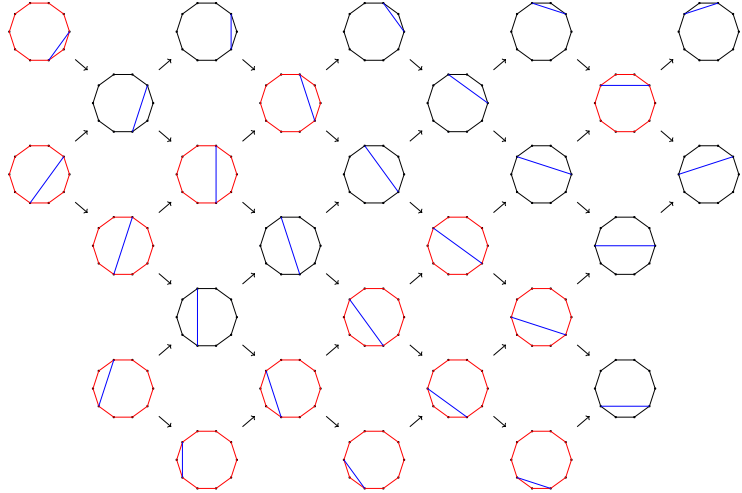
 \begin{adjustbox}{scale=0.4,center}
\begin{tikzcd}[ampersand replacement=\&, row sep = tiny , column sep = small]
\gonred{1}{A}{I}{}\arrow[rd]\&\& \gon{1}{B}{J}{}\arrow[rd] \&\& \gon{1}{C}{A}{}\arrow[rd]  \&\& \gon{1}{D}{B}{}\arrow[rd] \&\&  \gon{1}{C}{E}{}\\
\& \gon{1}{B}{I}{}\arrow[rd]\arrow[ru] \&\& \gonred{1}{C}{J}{}\arrow[rd]\arrow[ru]\&\& \gon{1}{D}{A}{}\arrow[rd]\arrow[ru]\&\& \gonred{1}{E}{B}{}\arrow[rd]\arrow[ru]\\
\gonred{1}{B}{H}{}\arrow[rd]\arrow[ru]\&\& \gonred{1}{C}{I}{}\arrow[rd]\arrow[ru]\&\& \gon{1}{D}{J}{}\arrow[rd]\arrow[ru]\&\& \gon{1}{E}{A}{}\arrow[rd]\arrow[ru]\&\& \gon{1}{F}{B}{}\\
\&  \gonred{1}{C}{H}{}\arrow[rd]\arrow[ru] \&\&\gon{1}{D}{I}{}\arrow[rd]\arrow[ru]\&\& \gonred{1}{E}{J}{}\arrow[rd]\arrow[ru]\&\& \gon{1}{F}{A}{}\arrow[ru]  \\
\&\& \gon{1}{D}{H}{}\arrow[rd]\arrow[ru] \&\& \gonred{1}{E}{I}{}\arrow[rd]\arrow[ru] \&\& \gonred{1}{F}{J}{}\arrow[rd]\arrow[ru] \\
\& \gonred{1}{D}{G}{}\arrow[ru]\arrow[rd] \&\& \gonred{1}{E}{H}{}\arrow[ru]\arrow[rd]  \&\& 
\gonred{1}{F}{I}{}\arrow[ru]\arrow[rd]\&\& \gon{1}{G}{J}{}\\ \&\& \gonred{1}{E}{G}{}\arrow[ru]\&\& \gonred{1}{F}{H}{}\arrow[ru]\&\& \gonred{1}{G}{I}{}\arrow[ru]
\end{tikzcd}
\end{adjustbox}
\\

Here, we have drawn  the polygons with sp-diagonals using red color. Hence, the AR-quiver $\Gamma(\mathcal{C}_{(T,F)})$  of the category $\mathcal{C}_{(T,F)}$ is identified with  the AR-quiver $\Gamma(\md_{sp}(k\P))$.
\end{example}  

\section{Associated subalgebra of the cluster algebra} Let $\P$ a poset of type $\mathbb{A}$ and let $Q^{F}$ be the quiver associated to $\P$ as in Proposition \ref{0}.  We denote by  $\mathcal{A}=\mathcal{A}(\bm x, Q)$ the cluster algebra associated to the initial seed $(\bm x, Q)$~\cite{fomin}. It is well known that the initial cluster variables in $\bm x$ correspond to the shift of indecomposable projectives in the cluster category (see \cite{BMRRT} ).  Let $\mathcal{A}(\P)$ be the subalgebra of $\mathcal{A}$ generated by the cluster variables $x_{\gamma}$ such that $\gamma$ is an sp-diagonal in the category $\mathcal{C}_{(T,F)}$ together with the cluster variables in the initial cluster $\bm x$. It is a natural to ask under which conditions we have $\mathcal{A}(\P)=\mathcal{A}$. A partial answer is given in Theorem \ref{subalgebra}.

\begin{lemma} Let $Q$ be a quiver of  tree type with $n$ vertices and let $\mathcal{A}=\mathcal{A}(\bm x,Q)$ be the cluster algebra associated to $Q$ with initial cluster $\bm x=\lbrace x_1,\dots,x_n\rbrace$.  If $\mathcal{A}'$ is the subalgebra of $\mathcal{A}$ generated by the cluster variables $x_1,\dots,x_n,x_{P_1},\dots, x_{P_n}$, where for all $i=1,\dots,n$, $x_{P_i}$ is the cluster variable associated to the indecomposable projective $kQ$-module $P_i$  in $\md kQ$, then $\mathcal{A}'=\mathcal{A}$. 
\end{lemma}  
  
\begin{proof} Because of \cite[Corollary 1.21]{BFZ} it suffices to show that $\mathcal{A}'$ contains the initial cluster $x_1,\dots,x_n$ as well as the $n$ cluster variables $x'_1,\dots,x'_n$ obtained from the initial cluster by a single mutation. We proceed by induction on the number $n$ of vertices in $Q$. The case $n=1$ is trivial. Now, let us consider $Q$ a tree with $n$ vertices, then $Q$ has $n-1$ arrows. Let $w$ be a leaf of $Q$ and define $Q'$ to be the full subquiver of $Q$ whose vertices are $Q_0\setminus\lbrace w\rbrace$. Then $Q$ is obtained from $Q$ by adding one vertex $w$ and one arrow $\alpha_w$ that starts or ends at $w$. We have two cases: either $(i)$ $\alpha_w:t\rightarrow w$ or $(ii)$ $\alpha_w:w\rightarrow t$ for some $t\in Q_0$. We recall the so-called exchange relation 
\begin{equation}\label{equ1}
x'_k x_k=p^{-}_k + p^{+}_k,
\end{equation}
defined for any vertex $k$ in $Q$, where $p^{-}_k=\prod_{\alpha:r\rightarrow k}x_r$ and $p^{+}_k=~\prod_{\beta:k\rightarrow r}x_r$. Here the product $p^{-}_k$ (respectively $p^{+}_k$) is taken over all arrows $\alpha\in Q_1$ (respectively $\beta\in Q_1$) that terminate (respectively start) in vertex $k$.  We shall proof that the variables $x'_w$ and $x'_{t}$ belong to $\mathcal{A}'$. In  case $(i)$, $w$ is a sink vertex and then $x'_w=x_{P_w}$. Hence, $x'_w\in\mathcal{A}'$. Additionally, following the knitting algorithm, we have that 
\begin{equation}\label{equ2}
x_{P_t}x_t=1+p^{-}_t \prod_{\beta: t\rightarrow r}x_{P_r},
\end{equation} 
 where the product is taken over all arrows  $\beta\in Q_1$ that start in vertex $t$.  We multiply  (\ref{equ2})  by $p_t^{+}$ and we obtain  $x_{P_t}x_t p_t^{+}=p_t^{+}+p^{-}_t \prod_{\beta: t\rightarrow r}x_rx_{P_r}$. Since $\alpha_w$ is an arrow from $t$ to $w$, then $x_{P_t}x_t p_t^{+}=p_t^{+}+p^{-}_t x_w x_{P_w}\delta$ 
 where the product $\delta= \prod_{\beta: t\rightarrow r\neq w}x_rx_{P_r}$ is taken over all arrows  $\beta\in Q_1$ that start in vertex $t$ and terminate in a vertex $r\neq w$.  Also,   $x_{P_t}x_t p_t^{+}=p_t^{+}+p^{-}_t(1+x_t)\delta$ because $x_w x_{P_w}=x'_w x_w= 1+x_t$. Since $x_{P_t},p^+_t\in\mathcal{A}'$ we have 

$$x_{P_t}p^+_t=\dfrac{p_t^{+}+p^{-}_t (1+x_t)\delta}{x_t}=\dfrac{p_t^{+}+p^{-}_t}{x_t}+p^{-}_t\delta\in\mathcal{A}'.$$

Since $p^{-}_t\delta$ belongs to $\mathcal{A}'$, equation (\ref{equ1}) implies $x'_t\in\mathcal{A}'$. In case $(ii)$, we have 
\begin{equation}\label{equ4}
x_{w}x_{P_w}=1+x_{P_t}.
\end{equation} 
Multiplying  (\ref{equ4}) by $x_t$ and  using (\ref{equ2}) we  deduce that $$x_{w}x_{P_w}x_t= x_t+1+p^{-}_t \prod_{\beta: t\rightarrow r}x_{P_r}.$$ 

Since $x_{P_w},x_t\in\mathcal{A}'$ then $$x_{P_w}x_t=\frac{x_t+1+p^{-}_t \prod_{\beta: t\rightarrow r}x_{P_r}}{x_w}\in\mathcal{A}'.$$
Moreover, since there is an arrow $\alpha_w$ from $w$ to $t$ then $x_w$ is a factor of $p_t^{-}$; thus, $x'_w=\frac{1+x_t}{x_w}\in\mathcal{A}'.$ Analogous to the proof of the case $(i)$, we can prove that $x'_t\in\mathcal{A}'$. \\
As a consequence of the hypothesis of induction on the quiver $Q'$ the variables $x'_s$ with vertex $s\neq t$ in $Q'_0$ belong to $\mathcal{A}'$. Thus, \cite[Corollary 1.21]{BFZ} implies the result.
\end{proof}

\begin{theorem}\label{subalgebra}
Let $\P$ be a poset of type $\mathbb{A}$ associated to the quiver $Q^{\emptyset}$ as in Proposition \ref{0} and let $\mathcal{A}(\P)$ be the subalgebra of $\mathcal{A}$  associated to $\P$. Then $\mathcal{A}(\P)=\mathcal{A}$.
\end{theorem}  
  
\begin{proof}
In this case, the poset $\P$ is viewed as the quiver $Q$ of type $\mathbb{A}$. Then, the subcategory $\mathcal{C}_{(T,F)}$ of $\mathcal{C}_T$ is given by $\star$-diagonals because $F=\emptyset$ and it is equivalent to the category $\md_{sp}kQ$ of socle-projective $kQ$-modules (see Theorem \ref{omega}).  By  Theorem \ref{TRalf} part (e) the indecomposable projectives in $\md kQ$ can be identified with diagonals $r^{+}(T)$ in the category $\mathcal{C}_T$ which are clearly $\star$-diagonals. Hence, the category     $\mathcal{A}(\P)$ contains the clusters variables described in the hypothesis of the above Lemma. Moreover,  $Q$ is a tree quiver. As a consequence, $\mathcal{A}(\P)=\mathcal{A}$.
\end{proof}

\section*{Acknowledgements}
The second author  would like to express his sincere gratitude to his advisors 
Agust\'{i}n Moreno Ca\~{n}adas and Hern\'{a}n Giraldo
for helpful suggestions and valuable discussions  on the topic.

\bibliographystyle{amsplain}

\end{document}